\documentclass[final,twoside,11pt]{entics}

\makeatletter

\usepackage{mathpartir,latexsym,stmaryrd,url,amsmath,amssymb,mathrsfs,mathtools,fontawesome,tikz-cd,bbding,listings,version,utfsym,graphicx,enumitem,phonetic,ifmtarg,pifont,tablefootnote,caption,subcaption,xfrac}
\usepackage[draft]{fixme}

\excludeversion{arxiv}

\fxusetheme{color}
\setlist{nolistsep}

\DeclareFontFamily{U}{FdSymbolA}{}
\DeclareFontFamily{U}{FdSymbolC}{}
\DeclareFontFamily{U}{FdSymbolF}{}
\DeclareFontShape{U}{FdSymbolA}{m}{n}{<-> s * FdSymbolA-Book}{}
\DeclareFontShape{U}{FdSymbolC}{m}{n}{<-> s * FdSymbolC-Book}{}
\DeclareFontShape{U}{FdSymbolF}{m}{n}{<-> s * FdSymbolF-Book}{}
\DeclareSymbolFont{fdsymbols}{U}{FdSymbolA}{m}{n}
\DeclareSymbolFont{fdarrows}{U}{FdSymbolC}{m}{n}
\DeclareSymbolFont{fddelimiters}{U}{FdSymbolF}{m}{n}
\DeclareMathSymbol{\Earrow}{\mathrel}{fdarrows}{8}
\DeclareMathSymbol{\Narrow}{\mathrel}{fdarrows}{9}
\DeclareMathSymbol{\Warrow}{\mathrel}{fdarrows}{10}
\DeclareMathSymbol{\Sarrow}{\mathrel}{fdarrows}{11}
\DeclareMathSymbol{\Nearrow}{\mathrel}{fdarrows}{12}
\DeclareMathSymbol{\Nwarrow}{\mathrel}{fdarrows}{13}
\DeclareMathSymbol{\Swarrow}{\mathrel}{fdarrows}{14}
\DeclareMathSymbol{\Searrow}{\mathrel}{fdarrows}{15}
\DeclareMathSymbol{\fdmapsto}{\mathrel}{fdarrows}{40}
\DeclareMathSymbol{\vdash}{\mathrel}{fdarrows}{224}
\DeclareMathSymbol{:}{\mathrel}{fdsymbols}{"02}
\DeclareMathSymbol{\later}{\mathord}{fdsymbols}{82}
\DeclareMathSymbol{\earlier}{\mathord}{fdsymbols}{84}
\DeclareMathSymbol{\medsquare}{\mathord}{fdsymbols}{117}
\DeclareMathSymbol{\timessquare}{\mathord}{fdsymbols}{124}
\DeclareMathSymbol{\dotsquare}{\mathord}{fdsymbols}{125}
\DeclareMathSymbol{\largesquare}{\mathord}{fdsymbols}{127}
\DeclareMathSymbol{\heartsuit}{\mathord}{fdsymbols}{183}
\DeclareMathSymbol{[}{\mathopen}{fddelimiters}{"12}
\DeclareMathSymbol{]}{\mathclose}{fddelimiters}{"18}
\DeclareMathSymbol{(}{\mathopen}{fddelimiters}{"00}
\DeclareMathSymbol{)}{\mathclose}{fddelimiters}{"06}
\DeclareMathDelimiter{\langle}{\mathopen}{fddelimiters}{"85}{fddelimiters}{"85}
\DeclareMathDelimiter{\rangle}{\mathclose}{fddelimiters}{"8B}{fddelimiters}{"8B}
\DeclareMathSymbol{/}{\mathord}{fddelimiters}{"A9}
\DeclareMathSymbol{|}{\mathord}{fddelimiters}{"B6}

\DeclareFontFamily{U}{min}{}
\DeclareFontShape{U}{min}{m}{n}{<-> udmj30}{}
\newcommand{\yon}{\!\text{\usefont{U}{min}{m}{n}\symbol{'210}}\!}

\DeclareFontFamily{U}{bskma}{\skewchar\font130 }
\DeclareFontShape{U}{bskma}{m}{n}{<->bskma10}{}
\DeclareSymbolFont{bskadd}      {U}         {bskma}{m}{n}
\DeclareMathSymbol{\hermitmatrix}         {\mathord}{bskadd}    {"F2}

\DeclareSymbolFont{txsymbolsC}{U}{txsyc}{m}{n}
\DeclareMathSymbol{\DiamonddotRight}{\mathrel}{txsymbolsC}{142}
\DeclareMathSymbol{\Diamonddotright}{\mathrel}{txsymbolsC}{134}
\DeclareMathSymbol{\Diamondright}{\mathrel}{txsymbolsC}{132}

\def\L{\mathcal{L}}
\def\M{\mathcal{M}}

\def\A{\mathscr{A}}
\def\B{\mathscr{B}}
\def\C{\mathscr{C}}
\def\D{\mathscr{D}}
\def\E{\mathscr{E}}

\def\cS{\mathcal{S}}

\def\@paren#1#2\endparen{\@ifmtarg{#2}{#1}{(#1#2)}}
\def\paren#1{\@paren #1 \endparen}
\def\F#1#2{\paren{#1}\mathord{\dotsquare} \paren{#2}}

\def\bigF#1#2{\mathord{\big\langle#1\big|}\@ifmtarg{#2}{}{#2\big\rangle}}
\def\preF#1#2{\paren{#1}\mathord{\medsquare} #2}
\def\U#1#2{\mathord{\langle#1\Vert}\@ifmtarg{#2}{}{#2\rangle}}

\def\bigU#1#2{\mathord{\big\langle#1\big\Vert}\@ifmtarg{#2}{}{#2\big\rangle}}

\def\p{\mathsf{p}}
\def\phat{\widehat{\mathsf{p}}}

\def\lexk{\mathcal{L}\mathit{ex}_\kappa}
\def\cat{\mathcal{C}\mathit{at}}
\def\gpd{\mathcal{G}\mathit{pd}}
\def\twocat{2\text{-}\mathcal{C}\mathit{at}}
\def\Chat{\mathrlap{\widehat{\phantom{C}}}\mathscr{C}}
\def\Ehat{\mathrlap{\widehat{\phantom{E}}}\mathscr{E}}

\def\Ghat{\widehat{G}}

\def\lock{\text{\faLock}}

\def\divby#1{\sfrac{}{#1}}
\def\keyby#1#2{#1\sfrac{}{#2}}
\def\cmpnt#1#2{\smash{#1}^{#2}}
\def\ocmpnt#1#2{\incl{#2}({#1})}

\def\Cunlock#1{\C_{#1}}

\def\Cradj#1{\C_{#1}^{\#}}

\def\reflect#1{\mathsf{L}^{#1}}
\def\incl#1{\mathbf{R}_{#1}}

\let\To\Rightarrow
\let\too\longrightarrow
\let\xto\xrightarrow

\def\ctx{\;\mathsf{ctx}}
\def\type{\;\mathsf{type}}

\def\sinister{\;\mathrm{sinister}}
\def\tangible{\;\mathrm{tangible}}
\def\issharp{\;\mathrm{sharp}}
\def\transparent{\;\mathrm{transparent}}

\def\ec{\diamond}

\def\op{^{\mathsf{op}}}
\def\coop{^{\mathsf{coop}}}

\def\flet#1#2#3#4{\mathsf{let}_{#1}\,#2 \leftarrow #3 \,\mathsf{in}\,#4}

\def\commacomp#1#2{(#1\downarrow (#2\circ -))}

\def\locks{\mathsf{locks}}
\def\Type{\mathsf{Type}}

\def\Tm{\mathrm{Tm}}
\def\Ty{\mathrm{Ty}}
\def\Tmb{\mathrm{Tm}^!}
\def\Tyb{\mathrm{Ty}^!}

\def\Tyhatb{\widehat{\mathrm{Ty}}{}^!}

\def\e{\mathrm{e}}
\def\f{\mathrm{f}}
\def\Tme{\mathrm{Tm}^{\mathrm{e}}}
\def\Tye{\mathrm{Ty}^{\mathrm{e}}}
\def\taue{\tau^{\mathrm{e}}}
\def\Tmf{\mathrm{Tm}^{\mathrm{f}}}
\def\Tyf{\mathrm{Ty}^{\mathrm{f}}}
\def\tauf{\tau^{\mathrm{f}}}
\def\Pr{\mathcal{P}}

\def\sPrl#1/#2{\mathrm{L}_{#2}\mathcal{P}^{\dDelta}_{\mathrm{inj}}(#1)}
\def\sPrlp#1/#2{\mathrm{L}_{#2}\mathcal{P}^{\dDelta}_{\mathrm{proj}}(#1)}

\def\taub{\tau^!}

\def\tauhatb{\widehat{\tau}{}^!}

\def\ce{\mathord{\,\triangleright\,}}
\def\tce{\mathord{\triangleright}}
\def\mce#1{\mathord{\,{\triangleright}^{#1}\,}}
\def\mnce#1#2{\mathord{\hermitmatrix^{#1}_{#2}}}

\def\leftdiv#1\by#2{#1 / #2}

\def\slice#1//#2{#1 \mathord{\sslash} #2}

\def\uparr{\mathord{\uparrow}}
\def\blank{\underline{\hspace{1.2ex}}}

\def\topos{\mathcal{T}\mathit{opos}}

\def\Set{\mathbf{Set}}
\def\sSet{\mathbf{sSet}}

\def\ladj#1{#1[#1^\dagger]}
\def\ladjs#1#2{#1[#2^\dagger]}
\def\LS{\ladjs{\L}{\cS}}
\def\radj#1{#1[\prescript{\dagger}{}{#1}]}

\def\sub#1\for#2\in#3{#3[#2 \mapsto #1]}
\def\Ps{\mathcal{P}\!\mathit{s}_{\mathit{colax}}}
\def\Psadj{\mathcal{P}\!\mathit{s}^{\mathit{adj}}}
\def\mod{\mathsf{mod}}
\def\shut#1#2{{#1}\mathord{\fdmapsto}\paren{#2}}
\def\open#1#2{\paren{#2}\raisebox{1pt}{\textbf{\scriptsize @}}#1}
\def\shutinv#1#2{{#1}\mathord{\fdmapsto}^{-1}#2}

\def\tcirc{\mathord{\circ}}

\def\blank{\underline{\hspace{1.3ex}}}

\DeclareFontFamily{U}{MnSymbolE}{}
\DeclareFontShape{U}{MnSymbolE}{m}{n}{
    <-6>  MnSymbolE5
   <6-7>  MnSymbolE6
   <7-8>  MnSymbolE7
   <8-9>  MnSymbolE8
   <9-10> MnSymbolE9
  <10-12> MnSymbolE10
  <12->   MnSymbolE12}{}
\newcommand{\mnquote}[1]{\usefont{U}{MnSymbolE}{m}{n}\char#1\relax}
\newcommand{\quinelquote}{\mathchoice{\raisebox{.1ex}{\rlap{\mnquote{'036}}\kern.25em}}
  {\raisebox{.1ex}{\rlap{\mnquote{'036}}\kern.25em}}
  {\raisebox{-.4ex}{\rlap{\mnquote{'036}}\kern.2em}}
  {\raisebox{-.7ex}{\rlap{\mnquote{'036}}\kern.2em}}}
\newcommand{\quinerquote}{\/\mathchoice{\raisebox{.1ex}{\kern.25em\llap{\mnquote{'043}}}}
  {\raisebox{.1ex}{\kern.25em\llap{\mnquote{'043}}}}
  {\raisebox{-.4ex}{\kern.2em\llap{\mnquote{'043}}}}
  {\raisebox{-.7ex}{\kern.2em\llap{\mnquote{'043}}}}}
\newcommand{\nameof}[1]{\quinelquote #1 \quinerquote}
\newcommand{\baseof}[1]{{\mathsf{V}_{#1}}}
\newcommand{\topof}[1]{\mathsf{E}_{#1}}

\numberwithin{equation}{section}
\numberwithin{subsubsection}{section}

\newcommand{\drpullback}[1][dr]{\ar[#1,phantom,near start,"\lrcorner"]}

\def\U#1#2{\@paren #1\endparen \mathord{\Diamonddotright} \paren{#2}}
\def\preU#1#2{\@paren #1\endparen \mathord{\Diamondright} #2}
\def\Cradj#1{\C_{#1^\dagger}}
\def\mnce#1#2{\mce{#1 \circ #2^\dagger}}

\def\bG{\boldsymbol\Gamma}
\def\bD{\boldsymbol\Delta}
\def\bth{\boldsymbol\theta}

\makeatother
\usepackage{enticsmacro}

\usepackage{aliascnt}
\usepackage{cleveref}
\def\defthm#1#2#3{%
  \newaliascnt{#1}{equation}
  \newtheorem{#1}[#1]{#2}
  \aliascntresetthe{#1}
  \crefname{#1}{#2}{#3}
}
\defthm{mythm}{Theorem}{Theorems}
\defthm{mythmsch}{Theorem-Schema}{Theorem-Schemas}
\defthm{mylem}{Lemma}{Lemmas}
\defthm{mycor}{Corollary}{Corollaries}
\defthm{myexmp}{Example}{Examples}
\defthm{mydefn}{Definition}{Definitions}
\defthm{myassum}{Assumption}{Assumptions}
\defthm{myrem}{Remark}{Remarks}
\newtheorem{unrem}{Remark}[section]

\crefname{figure}{Figure}{Figures}
\crefformat{enumi}{(#2#1#3)}

\renewenvironment{theorem}{\vspace{-\lastskip}\par\addvspace{.6pc plus
    .2pc minus .1pc}\begin{mythm}}{\end{mythm}\par\addvspace{.6pc plus
    .2pc minus .1pc}}
\renewenvironment{lemma}{\vspace{-\lastskip}\par
  \addvspace{.6pc plus .2pc minus .1pc}\begin{mylem}}
  {\end{mylem}\par\addvspace{.6pc plus .2pc minus .1pc}}
\renewenvironment{corollary}{\vspace{-\lastskip}\par \addvspace{.6pc
    plus .2pc minus .1pc}\begin{mycor}}{\end{mycor}\par\addvspace{.6pc
    plus .2pc minus .1pc}}
\renewenvironment{definition}{\vspace{-\lastskip}\par \addvspace{.6pc
    plus .2pc minus
    .1pc}\begin{mydefn}\rm}{\end{mydefn}\par\addvspace{.6pc plus .2pc
    minus .1pc}}
\renewenvironment{example}{\vspace{-\lastskip}\par \addvspace{.6pc plus
    .2pc minus .1pc}\begin{myexmp}\rm}{\end{myexmp}\par\addvspace{.6pc
    plus .2pc minus .1pc}}
\renewenvironment{remark}{\vspace{-\lastskip}\par
  \addvspace{.6pc plus .2pc minus .1pc}\begin{myrem}\rm}
{\end{myrem}\par\addvspace{.6pc plus .2pc minus .1pc}}
\newenvironment{unremark}{\vspace{-\lastskip}\par
  \addvspace{.6pc plus .2pc minus .1pc}\begin{unrem}\rm}
{\end{unrem}\par\addvspace{.6pc plus .2pc minus .1pc}}
\newenvironment{theoremschema}{\vspace{-\lastskip}\par\addvspace{.6pc plus
    .2pc minus .1pc}\begin{mythmsch}}{\end{mythmsch}\par\addvspace{.6pc plus
    .2pc minus .1pc}}
\newenvironment{assumption}{\vspace{-\lastskip}\par\addvspace{.6pc plus
    .2pc minus .1pc}\begin{myassum}}{\end{myassum}\par\addvspace{.6pc plus
    .2pc minus .1pc}}

\def\conf{MFPS 2023} 	%%Fill in the acronym for your conference (with year)
\def\myconf{mfps39}
\volume{3}			%Fill in the ENTICS volume number here
\def\volu{3}			% and here
\def\papno{18}			%Fill in your paper number here
\def\lastname{Shulman} %Lastnames appear in the running header 
\def\runtitle{Semantics of multimodal adjoint type theory} %Short title appears in the running header on even pages. 
\def\copynames{M. Shulman}    %% Fill in the first initial and last name of the authors
\def\mydoi{12300}
\begin{document}
%%%Note the beginning and end of the frontmatter section that starts here%%%%%
\begin{frontmatter}
  \title{Semantics of Multimodal Adjoint Type Theory\thanksref{ALL}} 						%%Title here and the
 \thanks[ALL]{This material is based upon work supported by the Air Force Office of Scientific Research under award number FA9550-21-1-0009.}   %%Text of \thanks[ALL} here..
 %%%%%%%%%%%%%%%%%%%%%%%%%%%%			This Thanks is optional.
  %%%%Now the author(s) names(s)%%%%%
  \author{Michael Shulman\thanksref{a}\thanksref{myemail}}	%%Note NO SPACE between 
    %%%Next come the addresses%%%%
   \address[a]{Department of Mathematics\\ University of San Diego\\				%or between \thanksrefs...
    San Diego, CA, USA}  							
   \thanks[myemail]{Email: \href{mailto:shulman@sandiego.edu} {\texttt{\normalshape
        shulman@sandiego.edu}}} 
   %%%Note: if both authors share same institution, only list the address once, after the second 
   %%%author. 
   %%%There also is a link from the first author to the co-author's address to show how to list 
   %%%affiliations to more than one institution, when needed. 
\begin{abstract} 
  We show that contrary to appearances, Multimodal Type Theory (MTT) over a 2-category $\mathcal{M}$ can be interpreted in any $\mathcal{M}$-shaped diagram of categories having, and functors preserving, $\mathcal{M}$-sized limits, without the need for extra left adjoints.
  This is achieved by a construction called ``co-dextrification'' that co-freely adds left adjoints to any such diagram, which can then be used to interpret the ``context lock'' functors of MTT.
  Furthermore, if any of the functors in the diagram have right adjoints, these can also be internalized in type theory as negative modalities in the style of FitchTT.
  We introduce the name Multimodal Adjoint Type Theory (MATT) for the resulting combined general modal type theory.
  In particular, we can interpret MATT in any finite diagram of toposes and geometric morphisms, with positive modalities for inverse image functors and negative modalities for direct image functors.
\end{abstract}
\end{frontmatter}

\section{Introduction}
\label{sec:introduction}

\emph{Modal type theories} involve type-forming operations, such as the classical $\Box$ (necessity) and $\lozenge$ (possibility), whose introduction and elimination rules modify the accessibility of previous hypotheses.
The increasing number of modal type theories has led to a need for general frameworks that can be instantiated to any new example, to avoid having to develop the metatheory of each new modal type theory from scratch.

After~\cite{ls:1var-adjoint-logic,lsr:multi}, each instantiation of a general modal type theory is determined by a 2-category
%\footnote{In~\cite{lsr:multi} a more general 2-categorical structure was used, enabling the subsumption of \emph{substructural} as well as modal type theories.
% In this paper we are concerned only with structural (dependent) modal type theories, so our mode theories will be plain 2-categories.}
$\M$, the ``mode theory''.
Its objects denote ``modes'', its morphisms generate modal operators relating types at different modes, and its 2-cells govern their interaction.
% In fact, in~\cite{ls:1var-adjoint-logic,lsr:multi} each morphism in $\M$ generates an \emph{adjoint pair} of modal operators, the left adjoint ``positive'' and the right adjoint ``negative''.
% This is particularly pleasing because models of type theory often arise from toposes, and functors between toposes often come in adjoint pairs (geometric morphisms).
However, the ``LSR'' theory of~\cite{ls:1var-adjoint-logic,lsr:multi} is only simply typed, % and efforts to generalize it to dependent type theory have been frustrating.
%Relatedly,
its definitional equality is ill-behaved, and it uses awkward global context operations. % that are ill-suited to implementation and metatheory.

The more recent frameworks MTT~\cite{gknb:mtt} and FitchTT~\cite{gckgb:fitchtt} resolve these problems: they are dependently typed, with a well-behaved definitional equality, and only ever extend the context; all indications suggest their implementability~\cite{gratzer:norm-mtt,sgb:mitten}.
%(The difference between the two is in polarity: the modalities of MTT are positive while those of FitchTT are negative.)
However, their na\"ive semantics requires the functors interpreting the modal operators to have additional left adjoints (``context locks''), which are not visible internally in the syntax.
% Thus, MTT or FitchTT with mode theory $\M$, which looks as though it is talking about an $\M$-shaped diagram of categories and nothing more, does not appear to have semantics in an \emph{arbitrary} $\M$-shaped diagram of categories, only those in which additional left adjoints exist.

We will show that this defect is, for the most part, only apparent.
Namely, from any suitable $\M$-shaped diagram of categories, we construct a new diagram whose functors all \emph{do} have left adjoints, enabling an interpretion of MTT and FitchTT.
Moreover, the \emph{types} in the new diagram are induced from the original ones, so this interpretation directly yields information about the original diagram.
We call this the \emph{co-dextrification}, since it makes existing functors into \emph{right} adjoints and has a ``cofree'' universal property.

%%%%%%%%%%%%%%%%%%%%%%%%%%%%%%

To explain the co-dextrification, consider first \emph{split-context} modal type theories, e.g.\ as in~\cite{pd:modal,shulman:bfp-realcohesion,zwanziger:natmod-comnd}.
As an example, let $\M$ have two objects and one morphism $\mu:p\to q$; we want to interpret modal type theory in a diagram of two categories and a functor $\C_\mu : \C_p\to\C_q$.
The split-context theory has ordinary $p$-judgments $\Gamma \vdash_p \mathcal{J}$, but $q$-judgments $\Gamma \mid \Delta \vdash_q \mathcal{J}$ with the context split into a $p$-part $\Gamma$ and a $q$-part $\Delta$, where $\Delta$ can depend on $\Gamma$.
We consider the types in $\Gamma$ to implicitly have $\C_\mu$ applied.
The modal rules then rearrange these context parts: \cref{fig:split-rule} shows the split-context introduction rule for the $\mu$-modality.

\begin{figure}
  \centering
  \hfill
  \begin{subfigure}[b]{0.3\textwidth}
    \centering
    \(\infer{\Gamma \vdash_p a : A}{\Gamma\mid\Delta \vdash_q \mathsf{mod}(a) : \F\mu A}\)
    \caption{The split-context rule}
    \label{fig:split-rule}
  \end{subfigure}
  \hfill
  % \begin{subfigure}[b]{0.3\textwidth}
  %   \centering
  %   \(\infer{\Gamma,\lock_\mu \vdash_p a : A}{\Gamma \vdash_q \mathsf{mod}(a) : \F\mu A}\)
  %   \caption{The context-lock rule}
  %   \label{fig:lock-rule}
  % \end{subfigure}
  % \hfill
  \begin{subfigure}[b]{0.3\textwidth}
    \centering
    \(\infer{\Gamma\divby\mu \vdash_p a : A}{\Gamma \vdash_q \mathsf{mod}(a) : \F\mu A}\)
    \caption{The division rule}
    \label{fig:div-rule}
  \end{subfigure}
  \hfill
  ~
  \caption{Comparison of modal introduction rules}
  \label{fig:cmp-rules}
\end{figure}

Following~\cite{zwanziger:natmod-comnd}, the $q$-contexts $(\Gamma\mid\Delta)$ suggest semantics in the comma category $\Chat_q = (\C_q\downarrow \C_\mu)$, whose objects are triples $(\Gamma,\Gamma\tce\Delta,\p_\Delta)$ where $\Gamma\in \C_p$, $\Gamma\tce\Delta\in \C_q$, and $\p_\Delta : \Gamma\tce\Delta\to \C_\mu(\Gamma)$.
This matches the split-context syntax, but can also be thought of as introducing a context lock functor in a ``universal'' way: the functor $\Chat_\mu : \C_p \to \Chat_q$ sending $\Gamma$ to $(\Gamma, \C_\mu(\Gamma), 1_{\C_\mu(\Gamma)})$ has a left adjoint $(-)\divby\mu$, defined by $(\Gamma, \Gamma\tce\Delta, \p_\Delta)\divby\mu = \Gamma$.
Thus, the rule in \cref{fig:split-rule} can also be written as in \cref{fig:div-rule}.

For general $\M$, there is no obvious way to split the context by restricting dependency.
Instead, the spiritual generalizations of split-context theories, sometimes called \emph{left-division} theories (e.g.~\cite{nvd:parametric-quant,mr:commuting-cohesions}), annotate each context variable with a morphism of $\M$ that is implicitly applied to it, and the modal rules modify these annotations.
In our simple example, each variable in a $q$-context is annotated with $\mu$ or $1_q$, and the operation $\Gamma\divby\mu$ deletes the $1_q$-annotated variables and uses the others to form a $p$-context.
For more general $\M$, when defining $\Gamma\divby\nu$, each annotated variable $x:^\mu A$ in $\Gamma$ is replaced by zero or more variables annotated by a family of morphisms $\varrho_i$ equipped with 2-cells $\alpha_i : \mu \To \nu\circ\varrho_i$ forming a \emph{left multi-lifting}, i.e.\ such that for any $\beta:\mu \To \nu\circ \sigma$ there is a unique $i$ and factorization of $\beta$ through $\alpha_i$ by some $\varrho_i \To \sigma$.

Unfortunately, a fully general $\M$ may not have all left multi-liftings.
In LSR, each rule application is instead allowed to choose \emph{any} morphism $\varrho$ with $\alpha:\mu \To \nu\circ \varrho$; the problems of LSR stem from the non-uniqueness of such a choice.
MTT and FitchTT solve this by delaying the choice of 2-cells, treating $(-)\divby\mu$ as a \emph{constructor} of contexts rather than an operation on them that computes.
(It is then sometimes written as $\Gamma.\lock_\mu$ or $\Gamma.\{\mu\}$, but I see no reason not to stick with $\Gamma\divby\mu$.)

\cref{fig:div-rule} shows $\F\mu{{(-)}}$ is ``right adjoint'' to $(-)\divby\mu$, so the semantics of these theories appears to require the modality functors to have left adjoints.
This contrasts with how we interpreted the split-context theory in a comma category, creating a \emph{new} left adjoint.
Some work~\cite{shulman:t2ttalk,licata:fibdtttalk} tried to generalize this by mimicking annotated contexts in semantics, but this was complicated and difficult.
Instead, we change perspective: rather than regarding an object of $(\C_q\downarrow \C_\mu)$ as an object $\Gamma\in \C_p$ together with an object $\Gamma\tce\Delta\in \C_q$ that \emph{depends on} $\C_\mu(\Gamma)$, we regard it as an object $\Delta\in \C_q$ together with a ``specified value of $\Delta\divby\mu$'' in $\C_p$, and a \emph{weakening} substitution from $\Delta$ to $\C_\mu(\Delta\divby\mu)$.
How a context is built from annotated types --- like the fact that it is built from types at all --- is a property of syntax that doesn't need to be reflected in semantics. %: an object of $\Chat_q$ is just an object of $\C_q$ that \emph{can be locked} in a specified way along any morphism $\mu:p\to q$.
We can now generalize to any $\M$: each $\Chat_q$ is an \emph{oplax limit} of the $\C_p$ over a slice 2-category.

It remains to specify how to \emph{extend} such a context by a type, i.e.\ how do we compute $(\Gamma, x :^\mu A)\divby\nu$ in terms of $\Gamma\divby\nu$ and $A$?
Instead of choosing \emph{one} pair $(\varrho,\alpha)$ as in LSR, or a universal family of them as in a multi-lifting theory, we use \emph{all} of them.
More precisely, we define $(\Gamma, x :^\mu A)\divby\nu$ to be the extension of $\Gamma\divby\nu$ by the \emph{limit} of $x :^\varrho A$ over all such $(\varrho,\alpha)$.
It is unclear whether this can be done syntactically, but semantically it is unproblematic.
When a left multi-lifting exists, this limit reduces to the \emph{product} of $x :^\varrho A$ over the elements of the multi-lifting.
And if there are \emph{no} such $\varrho$, the limit is a terminal object and $A$ is simply deleted, as happens to $\Delta$ in \cref{fig:split-rule}.

This is the essential idea of co-dextrification.
It is formally similar to Hofmann's ``right adjoint splitting''~\cite{hofmann:ttinlccc} for strict pullbacks, suggesting it can similarly be regarded as a sort of \emph{coherence theorem}.

The co-dextrification does require each $\C_p$ to have, and each $\C_\mu$ to preserve, limits of the size of $\M$.
This is unproblematic if $\M$ is finite, but modal operators often come in adjoint pairs (e.g.\ as geometric morphisms of topoi), and as soon as $\M$ contains a generic adjunction it is infinite.
Fortunately, if some $\C_\mu$ has a right adjoint, that adjoint automatically lifts to a \emph{dependent right adjoint} of $\Chat_\mu$.
Thus, it suffices to apply co-dextrification over a smaller 2-category $\L$ that generates $\M$ by adding some right adjoints.

The resulting type theory represents the morphisms in $\L$ by positive modalities as in MTT, but their right adjoints by negative modalities as in FitchTT.
(For a particular $\L$, such a combination appeared in~\cite{cavallo:thesis}.)
The positive elimination rules also restrict which morphisms of $\M$ can appear as ``framings'': this would be problematic for internalizing functoriality, except for the stronger elimination rule of the negative modalities.
We call this theory \textbf{Multimodal Adjoint Type Theory (MATT)}.
If we regard $\L$, rather than $\M$, as the fundamental parameter of MATT, then it restores the symmetry of~\cite{ls:1var-adjoint-logic,lsr:multi} in which each morphism (of $\L$) generates a positive/negative pair of modalities that are automatically adjoint.

\begin{ack}
  I am extremely grateful to Daniel Gratzer, for many long and illuminating conversations about modal type theories, for many concrete suggestions about MATT (including the name), and for careful reading and bugfixes.
  Dan Licata also contributed useful ideas to some of these conversations.
\end{ack}

\section{Multimodal Adjoint Type Theory}
\label{sec:matt}

For a 2-category $\M$ we write its objects as $p,q,r,s,\dots$, its morphisms as $\mu,\nu,\varrho,\sigma,\dots$, and its 2-cells as $\alpha,\beta,\dots$.
We use $\circ$ for both composition of morphisms and vertical composition of 2-cells, and write $\mu\triangleleft \beta$ and $\alpha\triangleright\nu$ for whiskering.
We will not use horizontal composition of 2-cells.

Although our semantics will have a mode theory with right adjoints added freely, it is simpler to formulate syntax using an arbitrary 2-category $\M$ equipped with placeholders for the necessary restrictions.

\begin{definition}
  An \textbf{adjoint mode theory} is a 2-category $\M$ equipped with four classes of morphisms in $\M$ called \textbf{tangible}, \textbf{sharp}, \textbf{transparent}, and \textbf{sinister}, such that
  % \begin{enumerate}
  % \item A class of morphisms called \textbf{tangible}.
  %   These can annotate variables in the context and in the domains of modal function-types, and give rise to covariant positive modal operators.
  % \item A class of morphisms called \textbf{transparent}.
  %   These can act as the ``framing'' modality in the elimination rule of other positive modal operators.
  % \item A class of morphisms called \textbf{sinister}.
  %   These are the morphisms that can give rise to contravariant negative modal operators.
  % \end{enumerate}
  % These data are required to satisfy the following axioms.
  \begin{itemize}%[resume]
  \item Every identity morphism is transparent and sharp.
  \item If $\mu:p\to q$ is sharp and $\nu:q\to r$ is transparent, then $\nu\circ\mu : p\to r$ is tangible.
    (Thus, every transparent or sharp morphism, and in particular every identity morphism, is tangible.)
  \item Every sinister morphism $\mu:p\to q$ has a right adjoint $\mu^\dagger:q\to p$ in $\M$, with unit $\eta_\mu : 1 \To \mu^\dagger \circ\mu$ and counit $\epsilon_\mu : \mu\circ \mu^\dagger \To 1$.
  \end{itemize}
\end{definition}

MATT over an adjoint mode theory $\M$ is MTT~\cite{gknb:mtt} over $\M$ with a few modifications.
We write $x :^\mu A$ in place of $x:(\mu\mid A)$, and $\F{\mu}{A}$ in place of $\langle \mu \mid A \rangle$.
We will show the most important MTT rules, but we omit technical details of substitutions.
We now list the substantive modifications.
\begin{enumerate}[label=(\arabic*)]
\item The modalities annotating variables in contexts must be tangible.
  Tangibility of identities yields ordinary type theories at each mode.
  The context rules are shown in \cref{fig:ctx}, along with a substitution rule that combines functoriality and naturality (the other substitution rules are more ordinary), and the variable-use rule in \cref{fig:var} along with the rule for substituting keys into variables.%
  \footnote{The latter is not fully precise, e.g.\ we have not defined the ``weakening'' substitution $\uparr^\alpha$.
    In the formal presentation of~\cite{gknb:mtt} there is only a zero-variable, to which can be applied substitutions involving 2-cell keys and weakening.}
  % We write just $x$ instead of $x^1$ when the 2-cell in the variable rule is an identity.
\begin{figure}
  \centering
  \begin{mathpar}
  \infer{ }{\ec_p \ctx_p}
  \and
  \infer{\Gamma \ctx_q \\ \mu : p\to q}{\Gamma\divby\mu \ctx_p}
  \and
  \infer{\Gamma \ctx_q \\ \mu : p\to q\tangible \\ \Gamma\divby\mu \vdash A \type_p}{(\Gamma, x:^\mu A) \ctx_q}  
  \and
  \infer{\Gamma \ctx_r \\ \mu : q\to r \\ \nu : p \to q}{\Gamma\divby\mu\divby\nu = \Gamma\divby{(\mu\circ\nu)}}
  \and
  \infer{\Gamma \ctx_p}{\Gamma\divby{1_p} = \Gamma}
  \and
    \infer{ \theta: \Gamma \to_q \Delta \\ \mu,\nu : p\to q \\ \alpha : \mu \To \nu}
    {\keyby\theta\alpha : \Gamma\divby\nu \to_p \Delta\divby\mu}
\end{mathpar}
  \caption{Contexts and substitutions in MATT}
  \label{fig:ctx}
\end{figure}
% \begin{figure}
%   \centering
%   \begin{mathpar}
%     \infer{ \Gamma \ctx_q \\ \mu,\nu : p\to q \\ \alpha : \mu \To \nu}
%     {\key^\alpha_\Gamma : (\Gamma,\lock_\nu) \to_p (\Gamma,\lock_\mu)}
%   \end{mathpar}
%   \caption{Key substitutions in MATT}
%   \label{fig:key}
% \end{figure}
\item The modalities $\mu$ that annotate domains of function-types $(x :^\mu A) \to B$ must be sharp.
  Sharpness of identities yields ordinary function-types, and tangibility of sharp morphisms is required for the formation and introduction rules.
  All the rules are shown in \cref{fig:func}.
\item The modalities $\mu$ that generate positive modal operators $\F\mu A$ must be sharp, and the ``framing'' modality in its elimination rule must be transparent.
  The rules for positive modal operators are shown in \cref{fig:pos-mod}.
  % We restate that rule here with the restriction:
  % \begin{mathpar}
  %   \infer{\mu : p\to q \issharp \\ \varrho : q\to r \transparent \\ \Gamma, x :^\varrho \F\mu A \vdash B \type_r \\ \Gamma,\lock_\varrho \vdash M_0 : \F\mu A \\ \Gamma, u :^{\varrho\circ\mu} A \vdash M_1 : B[\mod_\mu(u)/x]}{\Gamma \vdash \flet{\varrho}{\mod_\mu(u)}{M_0}{M_1} : B[M_0/x] }
  % \end{mathpar}
  The elimination rule requires both transparent morphisms, and composites of transparent and sharp morphisms, to be tangible.
  % Transparency of identities ensures we have a ``trivially framed'' elimination rule.
\item Every sinister morphism generates a \emph{negative} modal operator.
  These are not in MTT.
  Their rules are shown in \cref{fig:neg-mod}; they simplify those of~\cite{gckgb:fitchtt} by using right adjoints instead of parametric ones.
\end{enumerate}

\begin{figure}
  \centering
  \begin{mathpar}
    \locks(\ec_p) = 1_p \and
    \and
    \locks(\Gamma, x :^\mu A) = \locks(\Gamma)
    \and
    \locks(\Gamma\divby\mu) = \locks(\Gamma)\circ \mu
    \and
    \infer{\alpha : \mu \To \locks(\Delta)}{\Gamma, x:^\mu A, \Delta \vdash x^\alpha : A[\uparr^\alpha]}
    \and
    \infer{\alpha : \mu \To \locks(\Delta)\circ \nu \\ \beta : \nu \To \varrho}
    {(\Gamma, x:^\mu A, \Delta)\divby\varrho \vdash x^\alpha[\keyby{1_{(\Gamma,x:^\mu A,\Delta)}}\beta] = x^{(\locks(\Delta) \triangleright \beta) \circ \alpha}}
  \end{mathpar}
  \caption{Variables in MATT}
  \label{fig:var}
\end{figure}

\begin{figure}
  \centering
\begin{mathpar}
  \infer{\mu:p\to q\issharp \\
         \Gamma\divby{\mu} \vdash A \type_p \\ \Gamma, x:^\mu A \vdash B \type_q}
  {\Gamma \vdash (x:^\mu A) \to B \type_q}
  \and
  \infer{\mu:p\to q\issharp \\
    \Gamma\divby{\mu} \vdash A \type_p \\ \Gamma, x:^\mu A \vdash b : B}
  {\Gamma \vdash (\lambda x.b) : (x:^\mu A) \to B}
  \and
  \infer{\mu:p\to q\issharp \\
    \Gamma \vdash f : (x:^\mu A) \to B \\ \Gamma\divby\mu \vdash a:A}
  {\Gamma \vdash f\,a : B[x\leftarrow a]}
  \and
  \infer{\mu:p\to q\issharp \\
    \Gamma, x:^\mu A \vdash b : B \\ \Gamma\divby\mu \vdash a:A}
  {\Gamma \vdash (\lambda x.b)\,a = b[x\leftarrow a] : B[x\leftarrow a]}
  \and
  \infer{%\Gamma \vdash f : \textstyle\prod_{x:^\mu A} B \\ \Gamma \vdash g : \textstyle\prod_{x:^\mu A} B \\
    \mu:p\to q\issharp \\
  \Gamma , x :^\mu A \vdash f\,x = g\, x : B}
  {\Gamma \vdash f = g : (x:^\mu A) \to B}
\end{mathpar}
  \caption{Modal function-types in MATT}
  \label{fig:func}
\end{figure}

\begin{figure}
  \centering
\begin{mathpar}
  \infer{\mu : p \to q\issharp \\ \Gamma\divby\mu \vdash A \type_p}{\Gamma \vdash \F\mu A \type_q}
  \and
  \infer{\mu : p \to q\issharp \\ \Gamma\divby\mu \vdash a : A}{\Gamma \vdash \mod_{\mu}(a) : \F\mu A}
  \and
  \infer{\mu : p\to q\issharp \\ \nu : q \to r\transparent \\ \Gamma\divby\nu \vdash d : \F\mu A\\\\
    \Gamma, y:^{\nu} \F\mu A \vdash B \type_r \\
    \Gamma, x:^{\nu\circ\mu}A \vdash b :B[y\leftarrow \mod_\mu(x)]}
  {\Gamma \vdash \flet{\nu}{\mod_\mu(x)}{d}{b} : B[y\leftarrow d]}
  \and
  \infer{\mu : p\to q\issharp \\ \nu : q \to r\transparent \\ \Gamma\divby{(\nu\circ\mu)} \vdash a : A\\\\
    \Gamma, y:^{\nu} \F\mu A \vdash B \type_r \\
    \Gamma, x:^{\nu\circ\mu}A \vdash b :B[y\leftarrow \mod_\mu(x)]}
  {\Gamma \vdash (\flet{\nu}{\mod_\mu(x)}{\mod_\mu(a)}{b}) = b[x\leftarrow a]}
\end{mathpar}
  \caption{Positive modalities in MATT}
  \label{fig:pos-mod}
\end{figure}

\begin{figure}
  \centering
  \begin{mathpar}
    \infer{\mu : p\to q\sinister \\
      \Gamma\divby{\mu^\dagger} \vdash A \type_q}{\Gamma \vdash \U\mu A \type_p}
    \and
    \infer{\mu : p\to q \sinister \\
      \Gamma\divby{\mu^\dagger} \vdash M:A}{\Gamma \vdash \shut\mu{M} : \U\mu A}
    \and
    \infer{\mu : p\to q \sinister \\
      \Gamma\divby\mu \vdash M : \U\mu A}{\Gamma \vdash \open\mu{M} : A[\keyby{1_\Gamma}{\epsilon_\mu}]}
    \and
    \infer{\mu : p\to q \sinister \\
      \Gamma\divby{(\mu\circ{\mu^\dagger})} \vdash M : A}{\Gamma \vdash \open\mu{\shut\mu{M}} = M[\keyby{1_\Gamma}{\epsilon_\mu}] : A[\keyby{1_\Gamma}{\epsilon_\mu}]}
    \and
    \infer{\mu : p\to q \sinister \\
      \Gamma\divby{\mu^\dagger} \vdash \open\mu{M[\keyby{1_\Gamma}{\eta_\mu}]} = \open\mu{N[\keyby{1_\Gamma}{\eta_\mu}]} : A}
    {\Gamma\vdash M = N : \U\mu A}
  \end{mathpar}
  \caption{Negative modalities in MATT}
  \label{fig:neg-mod}
\end{figure}

\begin{arxiv}
\begin{unremark}
  Nothing in syntax mandates that the same class of mode morphisms be allowed to annotate domains of function-types and to form positive modalities (the sharp maps), but these classes coincide in all the models I know of.
  In co-dextrification models, the sharp maps also coincide with the tangible ones.
\end{unremark}
\end{arxiv}

\begin{remark}
  If $\mu$ is both sharp and sinister, the formation and introduction rules of $\U\mu A$ are identical to those of $\F{{\mu^\dagger}} A$.
  Daniel Gratzer has shown that $\U\mu A$ actually satisfies all the rules of $\F{{\mu^\dagger}} A$, while conversely if $\mu$ is transparent then $\F{{\mu^\dagger}} A$ satisfies all the rules of $\U\mu A$ except definitional $\eta$-conversion.
\end{remark}

\begin{arxiv}
\begin{unremark}\label{rmk:dra}
  We have written the $\eta$-rule for negative modalities in the form that would presumably be used by a bidirectional conversion-checking algorithm.
  However, an equivalent way to state it is:
  \begin{mathpar}
    \infer{\Gamma\vdash M : \U\mu{A}}{\Gamma \vdash M = \shut\mu{\open\mu{{M[1_\Gamma\divby{\eta_\mu}]}}} : \U\mu A}
  \end{mathpar}
  Then we can see that, as observed in~\cite{gckgb:fitchtt}, in the presence of the formation and introduction rules for $\U\mu A$, the other rules are equivalent to making the introduction rule invertible, which is to say that $\U\mu{-}$ is a dependent right adjoint~\cite{bcmmps:dep-radj} of $\lock_{\mu^\dagger}$.
  Given the other rules, the inverse of introduction is
  \begin{mathpar}
    \infer{\Gamma\vdash M : \U\mu{A}}{\Gamma\divby{\mu^\dagger} \vdash \open\mu{{M[1_\Gamma\divby{\eta_\mu}]}} : A}
  \end{mathpar}
  The $\beta$- and $\eta$-rules imply directly that this is an inverse.
  On the other hand, assuming the introduction rule to be invertible with inverse written $\shutinv{\mu}{M}$, say, we can define the elimination rule by
  \begin{mathpar}
    \infer*{
      \infer*{\Gamma,\lock_\mu \vdash M:\U\mu{A}}{\Gamma,\lock_\mu,\lock_{\mu^\dagger} \vdash \shutinv{\mu}{M} : A}}
    {\Gamma \vdash (\shutinv{\mu}{M})[1_\Gamma\divby{\epsilon_\mu}] : A[1_\Gamma\divby{\epsilon_\mu}]}.
  \end{mathpar}
  The $\beta$- and $\eta$-rules then follow.
\end{unremark}
\end{arxiv}

The flexibility in choosing the tangible, sharp, transparent, and sinister morphisms allows us to compare MATT easily to other modal type theories.
\begin{enumerate}
\item If $\M$ is any 2-category, and we take all morphisms to be tangible, sharp, and transparent, but none to be sinister, then MATT reduces to MTT.
\item For any 2-category $\L$, let $\M = \radj{\L}$ be obtained by formally adjoining a \emph{left} adjoint $\prescript{\dagger}{}{\mu}$ to each $\mu$ in $\L$.
  We take only identities to be tangible, sharp, and transparent, and the sinister morphisms to be these left adjoints $\prescript{\dagger}{}{\mu}$; then MATT reduces to FitchTT~\cite{gckgb:fitchtt} over $\L$ with actual left adjoints.
\item The closest match with theories such as~\cite{lsr:multi,shulman:bfp-realcohesion,mr:commuting-cohesions} occurs when $\M = \ladj{\L}$ is obtained by formally adjoining a \emph{right} adjoint $\mu^\dagger$ to each morphism $\mu$ of $\L$.
  In this case we take the tangible, sharp, and sinister morphisms to be the image of $\L$ in $\ladj{\L}$; thus all the modal operators come in adjoint pairs.

  \quad Different theories make different choices about transparency: in~\cite{shulman:bfp-realcohesion} only identities are transparent, while in~\cite{mr:commuting-cohesions} the transparent morphisms are also the image of $\L$.
  But in fact, if a morphism is both sinister and tangible, then it ``might as well'' be transparent, in that elimination rules with it as framing can be deduced from those with identity framing; the proof follows~\cite[Lemma 5.1]{shulman:bfp-realcohesion}.\label{item:ldiv}
\end{enumerate}

\medskip\noindent
Our semantics in the co-dextrification will apply to the following case.

\begin{example}\label{eg:ls}
  Let $\L$ be any 2-category and $\cS$ a class of morphisms in it, and let $\M = \LS$ be the result of freely adjoining a right adjoint $\mu^\dagger$ for every morphism $\mu$ in $\cS$.
  %\footnote{These adjoints are not strictly functorial: we have $(\mu\circ\nu)^\dagger \cong \nu^\dagger \circ \mu^\dagger$, but this is not an equality.}
  We identify $\L$ with its image in $\LS$.
  We take this image $\L$ to be the transparent morphisms, $\cS$ to be the sinister morphisms, and the tangible and sharp morphisms to be those that are isomorphic
  %\footnote{Since $(1_p)^\dagger$ is only isomorphic to $1_p$, allowing isomorphs or some such modification is necessary to ensure that all morphisms in $\L$, and hence all identities, are tangible.}
  to one of the form $\mu\circ \nu^\dagger$ where $\mu\in\L$ and $\nu\in\cS$.
%
% This choice of $\cS$ as sinister morphisms is obvious, while
This choice of tangible and sharp morphisms appears necessitated by our semantics (see \cref{thm:lw-modal}), and $\L$ is then the largest class of transparent morphisms satisfying the composition axiom.
\end{example}

\begin{assumption}\label{assume:ls}
  % Henceforth, if $(\L,\cS)$ are given,
  We always consider $\LS$ to be an adjoint mode theory as in \cref{eg:ls}.
\end{assumption}

\begin{arxiv}
  \begin{unremark}
    Of course, a model of MATT over some $\M$ remains a model if we shrink any of the three classes of morphisms.
    In particular, the co-dextrification also models MATT instantiated as in~\cref{item:ldiv} above, and this is the most common use case.
    We treat the more general class of tangible morphisms $\mu\circ \nu^\dagger$, which is not much more difficult semantically, out of a desire for the greatest reasonable generality.

    I only know of one application of this generality.
    Specifically, in MTT the positive modalities and their introduction and elimination rules can be given internal modal function-types involving universes:
    \begin{align*}
      \F\mu{\blank} &: (A :^\mu \Type_p) \to \Type_q\\
      \mod_\mu &: \{A :^\mu \Type_p\} (a :^\mu A) \to \F\mu A\\
      \mathsf{let}_{\nu,\mu} &:
      % \begin{multlined}[t]
      \{A :^\mu \Type_p\} (B : (y :^\nu \F\mu A) \to \Type_r)  (b : (x :^{\nu\circ\mu} A) \to B(\mod_{\mu}(x))) (d :^{\nu} \F\mu A) \to B\,d
      % \end{multlined}
    \end{align*}
    But to do the same for the negative modalities in MATT requires more general annotations of this sort.
    \begin{align*}
      \U\mu{\blank} &: (B :^{\mu^\dagger} \Type_q) \to \Type_p\\
      \shut{\mu}{\blank} &: \{B :^{\mu^\dagger} \Type_q\}(b :^{\mu^\dagger} B) \to \U\mu B\\
      \open{\mu}{\blank} &: \{B :^{\mu\circ\mu^\dagger} \Type_p\}(c :^{\mu} \U\mu B) \to B^{\epsilon_\mu}.
    \end{align*}
  \end{unremark}
\end{arxiv}

\begin{example}\label{eg:2ltt}
  We can regard Two-Level Type Theory~\cite{acks:2ltt} as an instance of MATT with two modes, $\f$ for (fibrant/inner) types and $\e$ for (non-fibrant/outer) exotypes, and an \emph{isomorphism} $\iota : \e\cong \f$.
  We let all the morphisms be tangible,
  % ; then types of either mode can appear freely in the context of the other, in a unique way.
  but we take only identities as sharp and transparent, and only the morphism $\iota : \e\to \f$ as sinister.
  Then $\U\iota-$ is the coercion from types to exotypes ($c$ in~\cite{acks:2ltt}), with a bijection between terms of types $A$ and $\U\iota A$.
  Allowing $\iota$ to be sharp would produce fibrant replacements $\F\iota A$, which are inconsistent~\cite[\S2.7]{acks:2ltt} with  univalence for fibrant types and UIP for exotypes.
  Inspecting the proof shows that the same conclusion would follow if we had modal function-types $(x :^\iota A) \to B$.
\end{example}

\begin{remark}\label{rmk:norm}
  It seems likely that normalization for MTT~\cite{gratzer:norm-mtt} extends to MATT.
  But to deduce decidability of type-checking from this requires decidability of equality for $\M$, whereas $\LS$ can fail to have decidable equality even if $\L$ does~\cite{dpp:undecidable-free-adjoint}.
  However, we can hope that $\LS$ will have decidable equality if $\L$ is, say, locally finite (this is true for for 1-categories~\cite{dpp:adj-adjoints}).
\end{remark}

\section{Natural models of MATT}
\label{sec:tt-cofree}

We now generalize the \emph{modal natural models} of~\cite{gknb:mtt} to MATT.
We first recall some definitions.
\begin{itemize}
\item A \textbf{natural model}~\cite{awodey:natmodels} is a representable morphism $\tau : \Tm \to \Ty$ in a presheaf category $\Pr\D$.
Thus for any $A\in \Ty(\Gamma)$ we have an object $\Gamma\ce A \in\D $, a morphism $\p_A : \Gamma\ce A\to \Gamma$, and a pullback square
\begin{equation}
  \begin{tikzcd}
    \yon(\Gamma\ce A) \ar[r] \ar[d,"\p_A"'] \drpullback & \Tm \ar[d,"\tau"] \\
    \yon(\Gamma) \ar[r,"A"'] & \Ty
  \end{tikzcd}\label{eq:comp-pb}
\end{equation}
where $\yon:\D \to \Pr\D$ denotes the Yoneda embedding.
A natural model is equivalent to a category with families.
We refer to $\Gamma\ce A$ as the \emph{comprehension} of $A$, and $\p_A$ as its \emph{type projection}.
\item A \textbf{modal context structure}~\cite[Definition 5.1]{gknb:mtt} is a 2-functor $\D : \M\coop\to\cat$ such that each $\D_p$ has a terminal object.
  We write its action on morphisms and 2-cells as $\D^\mu$ and $\D^\alpha$ respectively.
\item A \textbf{modal natural model}~\cite[Definition 5.4]{gknb:mtt} is a modal context structure $\D$ with a morphism $\tau_p : \Tm_p \to \Ty_p$ in each presheaf category $\Pr\D_p$, such that for any $\mu:p\to q$ in $\M$, the transformation $(\D^\mu)^* \tau_p$ is representable in $\Pr\D_q$.
  (Taking $\mu=1_p$, this implies that each $\D_p$ is a natural model.)
  We write the comprehension of $A \in \tau_p(\D^\mu(\Gamma))$ as $\p^\mu_A : \Gamma\mce{\mu}A \to \Gamma$, and write $\Gamma\mce{1}A$ as $\Gamma\ce A$.
\end{itemize}

\begin{arxiv}
  A modal natural model provides the structure of contexts, substitutions, and variables in MTT.
  For MATT, we simply impose the tangibility restriction.
\end{arxiv}

\begin{definition}
  Let $\M$ be an adjoint mode theory.
  A modal context structure $\D : \M\coop\to\cat$ is an \textbf{adjoint modal natural model} if we have a morphism $\tau_p : \Tm_p \to \Ty_p$ in each $\Pr\D_p$ such that $(\D^\mu)^* \tau_p$ is representable for all \emph{tangible} $\mu$.
  (Since identities are tangible, each $\D_p$ is still a natural model.)
\end{definition}

\begin{arxiv}
  We now similarly modify and unpack the modal type-formers of MTT.
  For later use, we also unpack the definitions more explicitly from the natural-model style given in~\cite{gknb:mtt}.
\end{arxiv}

\begin{definition}({See \cite[\S5.2.1]{gknb:mtt}})
  A \textbf{$\Pi$-structure} on an adjoint modal natural model $\D$ consists of, for any \emph{sharp} $\mu:p\to q$, and any $\Gamma\in \D_q$ and $A\in \Ty_p(\D^\mu(\Gamma))$ with $B\in \Ty_q(\Gamma\mce{\mu} A)$, a type $\Pi(A,B)\in \Ty_q(\Gamma)$ such that $\Gamma\ce\Pi(A,B)$ is a pushforward of $\Gamma\mce{\mu} A \ce B$ along $\p_A : \Gamma\mce{\mu} A \to A$, all natural in $\Gamma$.
\end{definition}

\begin{arxiv}
  The notion of ``equipped with the structure of a pushforward'' is cleanly expressed in the language of~\cite{weber:poly-pb} by saying we have a \emph{distributivity pullback}
  \[\begin{tikzcd}
      \bullet \ar[r] \ar[d] \drpullback & \Gamma \mce{\mu} A \ce B \ar[r,"\p_B"] & \Gamma\mce{\mu} A \ar[d,"\p_A"] \\
      \Gamma\ce\Pi(A,B)\ar[rr] & {} & \Gamma.
    \end{tikzcd}
  \]
  This means that this pullback square is terminal among such pullback squares with $\p_A$ and $\p_B$ fixed (known as ``pullbacks around $(\p_A,\p_B)$'').
\end{arxiv}

\begin{definition}({See~\cite[\S5.2.2]{gknb:mtt}})\label{defn:pos}
  An adjoint modal natural model $\D$ has \textbf{positive modalities} if for any \emph{sharp} $\mu:p\to q$ we have:
  \begin{enumerate}
  \item For any $\Gamma\in \D_q$ and $A\in \Ty_p(\D^\mu(\Gamma))$, we have a type $\F\mu A \in \Ty_q(\Gamma)$ and a map $j^\mu_{\Gamma,A} : \Gamma \mce{\mu} A \to \Gamma\ce (\F{\mu}{A})$ over $\Gamma$,
    % making a commutative triangle:
    % \[
    %   \begin{tikzcd}[column sep=small]
    %     \Gamma \mce{\mu} A \ar[dr] \ar[rr,dashed,"j^\mu_{\Gamma,A}"] & & \Gamma\ce (\F{\mu}{A}) \ar[dl] \\
    %     & \Gamma
    %   \end{tikzcd}
    % \]
    all varying naturally in \(\Gamma\).\label{item:pos-form-intro}
    % A commutative square:
    % \[
    %   \begin{tikzcd}
    %     \lock_\mu^* \Tm_p \ar[r,"j^{\mu}"] \ar[d] & \Tm_q \ar[d,"\tau_q"]\\
    %     \lock_\mu^* \Ty_p \ar[r,"\F\mu{}"'] & \Ty_q.
    %   \end{tikzcd}
    % \]
  \item For any \emph{transparent} $\varrho : q\to r$ and $\Gamma\in \D_r$ with $A\in \Ty_p(\D^{\varrho\circ\mu}(\Gamma))$, define the dashed map $\ell$ below by the universal property of pullbacks and full-faithfulness of $\yon$:
    \[\small
      \begin{tikzcd}
        \yon(\Gamma \mce{\varrho\circ\mu} A) \ar[r,dashed,"\yon(\ell)"] \ar[d] &
        \yon(\Gamma\mce{\varrho} (\F\mu A)) \ar[r] \ar[d]\drpullback &
        (\D^\varrho)^* \Tm_q \ar[d] \\
        \yon(\Gamma) \ar[r,equals] &
        \yon(\Gamma) \ar[r,"\F\mu A"'] &
        (\D^\varrho)^* \Ty_q
      \end{tikzcd}
      \; = \;
      \begin{tikzcd}
        \yon(\Gamma \mce{\varrho\circ\mu} A) \ar[r] \ar[d] \drpullback &
        (\D^{\varrho\circ\mu})^* \Tm_p  \ar[d] \ar[rr,"(\D^\varrho)^*(j^\mu)"] &&
        (\D^\varrho)^* \Tm_q \ar[d] \\
        \yon(\Gamma) \ar[r,"A"'] &
        (\D^{\varrho\circ\mu})^*\Ty_p \ar[rr,"(\D^\varrho)^*(\F{\mu}{-})"'] &&
        (\D^\varrho)^* \Ty_q.
      \end{tikzcd}
    \]
    Then for any commutative square as below there is a chosen diagonal filler, natural in $\Gamma$:
    % , we require that for any $B \in \Ty_r(\Gamma\mce{\varrho} (\F\mu A))$, and any map making the top of the following square commute, there exists a diagonal filler that varies naturally in $\Gamma$:
    \label{item:pos-elim-comp}
    \begin{equation*}
      \begin{tikzcd}
        \Gamma\mce{\varrho\circ\mu} A \ar[r] \ar[d,"\ell"'] &
        \Gamma\mce{\varrho} (\F\mu A) \ce B \ar[d,"\p_B"] \\
        \Gamma\mce{\varrho} (\F\mu A) \ar[r,equals] \ar[ur,dashed] &
        \Gamma\mce{\varrho} (\F\mu A)
      \end{tikzcd}%\label{eq:pos-elim-comp}
    \end{equation*}
  \end{enumerate}
\end{definition}

% (In~\cite[\S5.2.2]{gknb:mtt} this is phrased in terms of lifting structures in presheaves.)

%Following the methodology of~\cite{lw:localuniv}, we observe that this kind of strictly stable structure can be constructed from weakly stable structure.
% As before, we will construct this by strictifying some more categorically natural structures.

\begin{definition}(See~\cite[Definition 4]{gckgb:fitchtt})
  An adjoint modal natural model $\D$ has \textbf{negative modalities} if for any \emph{sinister} $\mu:p\to q$, the functor $\D^{\mu^\dagger}$ has a \emph{dependent right adjoint}~\cite{bcmmps:dep-radj}, i.e.\ there is a pullback square
\[
  \begin{tikzcd}
    (\D^{\mu^\dagger})^* \Tm_q \ar[r] \ar[d,"(\D^{\mu^\dagger})^*\tau_q"'] \drpullback & \Tm_p\ar[d,"\tau_p"] \\
    (\D^{\mu^\dagger})^*\Ty_q \ar[r] & \Ty_p.
  \end{tikzcd}
\]
\end{definition}

\begin{example}\label{eg:2ltt-model}
  Let $\M$ be the adjoint mode theory for Two-Level Type Theory from \cref{eg:2ltt}, and let $\C$ be a \emph{two-level model} as in~\cite[Definition 2.8]{acks:2ltt}.
  If we ignore universes, this means it has two natural models $\tauf : \Tmf \to \Tyf$ and $\taue : \Tme \to \Tye$, and that $\tauf$ is a pullback of $\taue$.
  % \begin{equation}
  %   \label{eq:2ltt-coe}
  %   \begin{tikzcd}
  %     \Tmf \ar[d,"\tauf"'] \ar[r] \drpullback & \Tme \ar[d,"\taue"] \\
  %     \Tyf \ar[r,"\mathrm{c}"'] & \Tye.
  %   \end{tikzcd}
  % \end{equation}
  Let $\D: \M\coop\to\cat$ be constant at $\C$, but where $\D_{\f}=\C$ is equipped with $\tauf$ while $\D_{\e} = \C$ is equipped with $\taue$.
  This is an adjoint modal natural model with negative modalities, since the assumption that $\tauf$ is a pullback of $\taue$ says exactly that the identity functor $(\C,\taue) \to (\C,\tauf)$ has a dependent right adjoint.
  % Thus, MATT faithfully represents the structure of 2LTT semantically as well as syntactically.
\end{example}

\section{Co-dextrification}
\label{sec:cofree}

\begingroup
\def\M{\L}

\begin{assumption}\label{assume:limits}
  For all of this section, let $\M$ be an arbitrary 2-category, let $\C:\M \to \cat$ be a pseudofunctor, and let $\kappa$ be an infinite regular cardinal such that $\M$ is $\kappa$-small, each category $\C_p$ has $\kappa$-small limits, and each functor $\C_\mu : \C_p\to \C_q$ preserves $\kappa$-small limits.
  Often, $\kappa$ will be $\omega$.
\end{assumption}

\begin{arxiv}
  \begin{unremark}
    In fact, it suffices if all the comma categories over which we take limits below, which are constructed from the hom-categories of $\M$, admit some initial functor from a $\kappa$-small category (when $\kappa=\omega$ this is called being L-finite~\cite{pare:sclim}).
    For instance, if all the relevant left liftings exist, then these categories have initial objects, hence satisfy this condition no matter how large the hom-categories of $\M$ are.
  \end{unremark}
\end{arxiv}

% \subsection{The oplax limit as a modal context structure}
% \label{sec:oplax-limit}

\begin{definition}
  For $r\in \M$, let $\slice\M// r$ denote the \textbf{lax slice 2-category}:
  \begin{itemize}
  \item Its objects are morphisms $\mu : p\to r$ in $\M$.
  \item Its morphisms from $\mu : p\to r$ to $\nu:q\to r$ are pairs $(\varrho:p\to q,\; \alpha : \mu \To \nu\circ\varrho)$.
  \item Its 2-cells from $(\varrho,\alpha)$ to $(\sigma,\beta)$ are 2-cells $\gamma :\varrho \To \sigma$ such that $(\nu\triangleleft\gamma)\circ \alpha = \beta$.
  \end{itemize}
\end{definition}

\noindent
By postcomposition, we have a 2-functor
\(\slice\M//- : \M \to \twocat\), with projection functors $\pi_r : \slice\M//r \to \M$. %, forming a cocone under this 2-diagram.

\begin{definition}\label{defn:chat}
  For $r\in\M$, let $\Chat_r$ denote the \textbf{oplax limit} of the $(\slice\M//r)$-shaped diagram $\C \circ \pi_r : \slice\M//r \to \cat$ in $\cat$.
  Thus, an object $\bG\in\Chat_r$ consists of:
  \begin{enumerate}
  \item For each $\mu:p\to r$ in $\M$, an object $\cmpnt\bG\mu \in \C_p$.
  \item For each $\varrho:p\to q$ and $\alpha : \mu \To \nu\circ\varrho$, a morphism
    \( \cmpnt\bG\alpha : \cmpnt\bG\nu \too \C_\varrho(\cmpnt\bG{\mu}) \)
    in $\C_q$.
    (The notation is abusive, as $\cmpnt\bG\alpha$ depends not just on $\alpha$ but on the decomposition of its codomain as a composite.)
  \item For $\alpha = 1_\mu : \mu \To \mu \circ 1_p$, we have $\cmpnt\bG{1_\mu} = 1_{\cmpnt\bG\mu}$.\label{item:chat-id}
    % If $\alpha : \mu \To \nu \circ 1$, then $\Gamma\smkey_\alpha = \Gamma \smlock_\alpha$.
  \item For $\alpha : \mu \To \nu\circ\varrho$ and $\beta : \nu \To \varpi \circ \sigma$, we have
    $\C_\sigma(\cmpnt\bG\alpha) \circ \cmpnt\bG\beta = \cmpnt\bG{(\beta\triangleright\varrho)\circ \alpha}$, modulo pseudofunctoriality.
 % the following diagram commutes:
    \begin{arxiv}
      \begin{equation*}
        \begin{tikzcd}[column sep=large]
          \C_\sigma(\bG^{\nu}) \ar[r,"\C_\sigma(\bG^\alpha)"] &
          \C_\sigma \C_\varrho(\bG^\mu)\\
          \bG^\varpi
          \ar[u,"\bG^\beta"]
          \ar[r,"\bG^{(\beta\triangleright\varrho)\circ \alpha}"'] &
          \C_{\sigma\circ\varrho}(\bG^{\mu}) \ar[u,"\cong"']
        \end{tikzcd}
      \end{equation*}
      In particular, taking $\varrho$ and $\sigma$ to be identities, for each $p\in \M$ we have a functor $\bG^p : \M(p,r)\op \to \C_p$.
    \end{arxiv}
  \item For $\alpha : \mu \To \nu \circ \varrho$ and $\beta : \varrho \To \sigma$, we have
    $\C_\beta(\cmpnt\bG\mu)\circ \cmpnt\bG\alpha = \cmpnt\bG{(\nu\triangleleft\beta)\circ\alpha}$.
    % the following diagram commutes:
    \begin{arxiv}
      \begin{equation*}
        \begin{tikzcd}
          & \C_\varrho(\bG^\mu) \ar[dr,"\C_\beta(\bG^\mu)"] \\
          \bG^\nu \ar[ur,"\bG^\alpha"] \ar[rr,"\bG^{(\nu\triangleleft\beta)\circ\alpha}"'] &&
          \C_\sigma(\bG^\mu)
        \end{tikzcd}
      \end{equation*}
    \end{arxiv}
  \end{enumerate}
  Similarly, a morphism $\bth : \bG \to \bD$ in $\Chat_r$ consists of:
  \begin{enumerate}[resume]
    \item For each $\mu:p\to r$, a morphism $\cmpnt\bth\mu : \cmpnt\bG\mu \to \cmpnt\bD\mu$.\label{item:chat-mor-1}
    \item For $\alpha : \mu \To \nu \circ\varrho$, we have
      $\C_\varrho(\cmpnt\bth\mu) \circ \cmpnt\bG\alpha = \cmpnt\bD\alpha \circ \cmpnt\bth\nu$.
      \label{item:chat-mor-2}
% the following diagram commutes:
      \begin{arxiv}
        \[
          \begin{tikzcd}
            \bG^\nu \ar[r,"\bG^\alpha"] \ar[d,"\bth^\nu"'] &
            \C_\varrho(\bG^\mu) \ar[d,"\C_\varrho(\bth^\mu)"] \\
            \bD^\nu \ar[r,"\bD^\alpha"'] &
            \C_\varrho(\bD^\mu)
          \end{tikzcd}
        \]
      \end{arxiv}
    \end{enumerate}
\end{definition}

\begin{lemma}\label{thm:chat-locks}
  The categories $\Chat_p$ are the action on objects of a modal context structure $\Chat : \M\coop \to \cat$.
  % We write its action on morphisms and 2-cells with $\lock$ and $\key$, respectively, both postfix.
\end{lemma}
\begin{proof}
  The functorial action is by composition: $\cmpnt{(\Chat^\mu(\bG))}{\nu} = \cmpnt\bG{\mu\circ\nu}$ and $\cmpnt{(\Chat^\beta(\bG))}\varrho = \cmpnt\bG{\beta\triangleright\varrho}$.
\end{proof}

For $\mu:p\to q$, write $\reflect \mu : \Chat_q \to \C_p$ for the functor defined by $\reflect \mu (\bG) = \cmpnt\bG\mu$.

\begin{lemma}\label{thm:chat-lim}
  Each $\Chat_p$ has $\kappa$-small limits, and each functor $\reflect\mu$ and $\Chat^\mu$ preserves them.
  Furthermore:
  \begin{enumerate}
  \item If each $\C_p$ has some shape of colimits, then so does each $\Chat_p$, and each $\reflect\mu$ and $\Chat^\mu$ preserves them.\label{item:chat1}
  \item If each $\C_p$ is locally cartesian closed or an elementary topos, so is each $\Chat_p$.\label{item:chat2}
  \item If each $\C_p$ is locally presentable, and each $\C_\mu$ is accessible, then each $\Chat_p$ is also locally presentable.\label{item:chat3}
  \item If each $\C_p$ is a Grothendieck topos, and each $\C_\mu$ is an inverse or direct image, then so is each $\Chat_p$.\label{item:chat4}
  \end{enumerate}
\end{lemma}
\begin{proof}
  The limits, and colimits in~\cref{item:chat1}, are defined pointwise.
  For~\cref{item:chat2}, an oplax limit is the category of coalgebras for a finitely continuous comonad on a product category (see~\cite{wraith:glueing} or~\cite[B3.4.6]{ptj:elephant}), and the stated properties are closed under products and such coalgebras (e.g.~\cite[A4.2.1]{ptj:elephant}).
  For~\cref{item:chat3}, by~\cite[Theorem 5.1.6]{mp:accessible} accessible categories and functors are closed under limits, and an accessible category is locally presentable if and only if it is cocomplete.
  For~\cref{item:chat4}, we use~\cref{item:chat2} and~\cref{item:chat3}, since Grothendieck topoi are the locally presentable elementary topoi~\cite[C2.2.8]{ptj:elephant}, and left and right adjoints are accessible.
\end{proof}

% \subsection{The right adjoints}
% \label{sec:right-adjoints}

\begin{lemma}\label{thm:smlock-adj}
  For $\varpi:r \to s$, the functor $\reflect{\varpi} : \Chat_s \to \C_r$ has a right adjoint, which we write $\incl\varpi$.
\end{lemma}
\begin{proof}
  Given $\Gamma\in \C_r$, we must first define $\cmpnt{(\incl\varpi\Gamma)}\nu \in \C_p$ for any $\nu:p\to s$.
  Let $\commacomp{\varpi}{\nu}$ be the category of pairs $(\sigma:r\to p,\, \beta : \varpi \To \nu\circ\sigma)$.
  % , and whose morphisms $(\sigma,\beta) \to (\sigma',\beta')$ are 2-cells $\gamma : \sigma \To \sigma'$ such that $(\nu\triangleleft\gamma)\circ\beta = \beta'$.
  Any such $(\sigma,\beta)$ induces an object $\C_\sigma(\Gamma) \in \C_p$; %, and any such $\gamma$ induces a morphism of these.
  we define
  \[ \cmpnt{(\incl\varpi\Gamma)}\nu = \lim_{(\sigma,\beta)\in \commacomp{\varpi}{\nu}} \C_\sigma(\Gamma). \]

  Now suppose given also $\varrho:p\to q$ and $\alpha : \mu \To \nu\circ \varrho$.
  Then $\cmpnt{(\incl\varpi\Gamma)}\alpha$ should be a morphism % (using limit-preservation of $\varrho$)
  \[ \cmpnt{(\incl\varpi\Gamma)}\nu = \lim_{(\sigma,\beta)\in \commacomp{\varpi}{\nu}} \C_\sigma(\Gamma)
    \too
    \lim_{(\sigma,\beta)\in \commacomp{\varpi}{\mu}} \C_{\varrho} \C_{\sigma}(\Gamma)
    \xto{\cong} \C_\varrho(\cmpnt{(\incl\varpi\Gamma)}\mu).
  \]
  If $(\sigma,\beta)\in \commacomp{\varpi}{\mu}$ indexes a factor $\C_\varrho\C_\sigma (\Gamma)$ of this codomain, then $(\varrho\circ \sigma, (\alpha \triangleright \sigma)\circ \beta) \in \commacomp{\varpi}{\nu}$, and the factor $\C_{\varrho\circ\sigma}(\Gamma)$ of the domain is isomorphic to $\C_\varrho\C_\sigma (\Gamma)$.
  Thus, this determines a map $\cmpnt{(\incl\varpi\Gamma)}\alpha$ between the limits.
  This defines $\incl\varpi\Gamma \in \Chat_s$.
  Now we observe that
  \[ \cmpnt{(\incl\varpi\Gamma)}{\varpi} = \lim_{(\sigma,\beta) \in \commacomp{\varpi}{\varpi}} \C_\sigma(\Gamma). \]
  Since $(1_r,1_{\varpi})\in \commacomp{\varpi}{\varpi}$, with $\C_{1_r}(\Gamma) \cong \Gamma$, there is a projection $\epsilon_\Gamma : \cmpnt{(\incl\varpi\Gamma)}{\varpi} \to \Gamma$.
  We claim this is a universal arrow from $\reflect\varpi$.
  For $\bD \in \Chat_s$, a map $\theta : \bD \to \incl\varpi\Gamma$ consists of, for any $\nu:p\to r$ and any $(\sigma,\beta)\in \commacomp{\varpi}{\nu}$, a morphism $\cmpnt\bth{\nu,(\sigma,\beta)} : \cmpnt\bD\nu \to \C_\sigma\Gamma$, such that for any $\alpha : \mu \To \nu\circ\varrho$ and $\beta : \varpi \To \mu\circ \sigma$:
  \[
    \begin{tikzcd}[row sep=large, column sep=large]
      \cmpnt\bD\nu \ar[r,"\cmpnt\bD\alpha"] \ar[d,"\C_\varrho(\cmpnt\bth\nu)" description]
      \ar[dd,bend right=60,"\cmpnt\bth{\nu,(\varrho\circ\sigma,(\alpha\triangleright\sigma)\circ\beta)}"'] &
      \C_\varrho(\cmpnt\bD\mu) \ar[d,"\cmpnt\bth\mu" description]
      \ar[dd,bend left=60,"\cmpnt\bth{\mu,(\sigma,\beta)}"]\\
      \cmpnt{(\incl\varpi\Gamma)}\nu \ar[r,"\cmpnt{(\incl\varpi\Gamma)}\alpha"] \ar[d] &
      \C_\varrho(\cmpnt{(\incl\varpi\Gamma)}\mu) \ar[d] \\
      \C_{\varrho\circ\sigma}\Gamma \ar[r,"\cong"'] & \C_\varrho \C_\sigma \Gamma.
    \end{tikzcd}
  \]
  Taking $\nu=\varpi$ and $\sigma=1_r$ with $\beta = 1_\varpi$ yields % map $\cmpnt\bth{\varpi,(1_r,1_\varpi)} : \cmpnt\bD\varpi \to \Gamma$.
  the composite $\cmpnt\bD\varpi \xto{\cmpnt\bth\varpi} \cmpnt{(\incl\varpi\Gamma)}\varpi \xto{\epsilon_\Gamma} \Gamma$.
  Moreover, if in the above condition we take $\mu=\varpi$ with $(\sigma,\beta) = (1_r,1_\varpi)$, then the left-hand vertical composite becomes $\cmpnt\bth{\nu,(\varrho,\alpha)}$, which is fully general; thus all the components of $\theta$ are determined by $\cmpnt\bth{\varpi,(1_r,1_\varpi)}$.

  Now, given $\vartheta : \cmpnt\bD\varpi \to \Gamma$, for any $\nu$ and $(\sigma,\beta)$ we have a composite
  \( \cmpnt\bD\nu \xto{\cmpnt\bD\beta} \C_\sigma(\cmpnt\bD\varpi) \xto{\C_\sigma(\vartheta)} \C_\sigma\Gamma \).
  The above compatibility condition follows from the axioms of \cref{defn:chat}, so we have a map $\bD \to \incl\varpi\Gamma$.
  Its underlying map $\cmpnt\bD\varpi \to \Gamma$ is
  \( \cmpnt\bD\varpi \xto{\cmpnt\bD{1_\varpi}} \C_{1_r}(\cmpnt\bD\varpi) \xto{\C_{1_\varpi}(\vartheta)} \C_{1_\varpi}\Gamma \cong\Gamma \),
  which is equal to $\vartheta$. % up to the coherence isomorphisms.
%  Thus, we have the requisite bijection. % to show the universal property of $\epsilon_\Gamma$.
\end{proof}

When $\varpi = 1_r$, we write $\reflect r = \reflect{1_r}$ and $\incl{r} = \incl{1_r}$.

\begin{lemma}\label{thm:up-ff}
  The functor $\incl r : \C_r \to \Chat_r$ is fully faithful.
\end{lemma}
\begin{proof}
  When $\varpi=1_r$, the element $(1_r, 1_{1_r})$ of $\commacomp{1_r}{1_r}$ is initial.
  Thus, the domain of $\epsilon_\Gamma$ is evaluation at that object, which is $\C_{1_r}(\Gamma) \cong \Gamma$.
  So $\epsilon$ is an isomorphism, hence $\incl r$ is fully faithful.
\end{proof}

\begin{lemma}\label{thm:up-mate}
  Let $\mu:p\to r$, $\nu:q\to r$, $\varrho:p\to q$, and $\alpha: \mu \To \nu\circ\varrho$.
  Then for any $\Gamma\in \C_p$ there is a map
  \( \ocmpnt\Gamma\alpha : \incl{\mu}(\Gamma) \to \incl{\nu}(\C_\varrho \Gamma),
  \),
  which varies naturally in $\Gamma$;
  it is the mate of $\cmpnt\bD\alpha : \cmpnt\bD\nu \to \C_\varrho(\cmpnt\bD\mu)$.\qed
\end{lemma}
% \begin{proof}
%   It is the mate of the transformation $\Delta\smlock_\nu \to \C_\varrho(\Delta\smlock_\mu)$ for $\Delta\in \Chat_r$.
%   % , which is natural by definition of the morphisms in $\Chat_r$.
% \end{proof}

\begin{lemma}\label{thm:chat-radj}
  For any $\varpi:r\to s$, the functor $\Chat^\varpi : \Chat_s \to \Chat_r$ has a right adjoint $\Chat_{\varpi} : \Chat_r \to \Chat_s$.
\end{lemma}
\begin{proof}
  Let $\bG \in \Chat_s$ and $\bD \in \Chat_r$.
  By definition, a morphism $\bth : \Chat^\varpi(\bG) \to \bD$ consists of components $\cmpnt\bth\mu : \cmpnt\bG{\varpi\circ\mu} \to \cmpnt\bD\mu$ for all $\mu:p\to r$ such that for any $\alpha : \mu\To \nu\circ\varrho$ the following diagram commutes:
  \begin{equation}
    \begin{tikzcd}
      \cmpnt\bG{\varpi\circ\nu} \ar[r,"\cmpnt\bG{\varpi\triangleleft\alpha}"] \ar[d,"\cmpnt\bth\nu"'] &
      \C_\varrho(\cmpnt\bG{\varpi\circ\mu}) \ar[d,"\C_\varrho(\cmpnt\bth\mu)"] \\
      \cmpnt\bD\nu \ar[r,"\cmpnt\bD\alpha"'] &
      \C_\varrho(\cmpnt\bD\mu)
    \end{tikzcd}\label{eq:chat-radj-condition}
  \end{equation}
  To give $\cmpnt\bth\mu$ is equivalent to give $\overline{\cmpnt\bth\mu}: \bG \to \incl{\varpi\circ\mu}(\cmpnt\bD\mu)$.
  We will define $\Chat_{\varpi}(\bD)\in \Chat_s$ as the limit of a diagram of objects $\incl {\varpi\circ\mu}(\cmpnt\bD\mu)$, so that a map $\bG \to \Chat_{\varpi}(\bD)$ is determined by maps $\overline{\cmpnt\bth\mu}$ satisfying a cone condition that is equivalent to~\eqref{eq:chat-radj-condition}.
  We start by writing down the naturality square for the transformation $\incl{\varpi\triangleleft \alpha}$ of \cref{thm:up-mate} at $\cmpnt\bth\mu$, and composing it with the adjunction unit $\bG \to \incl{\varpi\circ\mu}(\cmpnt\bG{\varpi\circ\mu})$:
  \begin{equation}
    \begin{tikzcd}
      \bG \ar[r] \ar[dr] 
      & \incl{\varpi\circ\mu}(\cmpnt\bG{\varpi\circ\mu}) \ar[rr,"\incl{\varpi\circ\mu}(\cmpnt\bth\mu)"] \ar[d,"\incl{\varpi\triangleleft \alpha}(\cmpnt\bG{\varpi\circ\mu})"] &&
      \incl{\varpi\circ\mu}(\cmpnt\bD\mu)\ar[d,"\incl{\varpi\triangleleft \alpha}(\cmpnt\bD\mu)"] \\
      & \incl{\varpi\circ\nu}(\C_\varrho(\cmpnt\bG{\varpi\circ\mu})) \ar[rr,"\incl{\varpi\circ\nu}(\C_\varrho(\cmpnt\bth\mu))"'] &&
      \incl{\varpi\circ\nu}(\C_\varrho(\cmpnt\bD\mu))
    \end{tikzcd}\label{eq:chat-radj-nat}
  \end{equation}
  We also transpose~\eqref{eq:chat-radj-condition} across $\reflect {\varpi\circ\nu} \dashv \incl{\varpi\circ\nu}$ to obtain an equivalent condition as at left below:
  \begin{equation}
    \begin{tikzcd}
      \bG \ar[r] \ar[d] &
      \incl{\varpi\circ\nu}(\C_\varrho(\cmpnt\bG{\varpi\circ\mu})) \ar[d] \\
      \incl{\varpi\circ\nu}(\cmpnt\bD\nu) \ar[r] &
      \incl{\varpi\circ\nu}(\C_\varrho(\cmpnt\bD\mu))
    \end{tikzcd}%\label{eq:chat-radj-condition-tr}
    \hspace{2cm}
    \begin{tikzcd}
      \bG \ar[r,"\overline{\cmpnt\bth\nu}"] \ar[d,"\overline{\cmpnt\bth\nu}"'] &
      \incl{\varpi\circ\mu}(\cmpnt\bD\mu) \ar[d] \\
      \incl{\varpi\circ\nu}(\cmpnt\bD\nu) \ar[r] &
      \incl{\varpi\circ\nu}(\C_\varrho(\cmpnt\bD\mu))
    \end{tikzcd}%\label{eq:chat-radj-condition-mod}
    \label{eq:chat-radj-condition-new}
  \end{equation}
  The left-bottom composite in~\eqref{eq:chat-radj-nat} is equal to the top-right composite
  at left in~\eqref{eq:chat-radj-condition-new}.
  Thus, we can replace this part of the square at left in~\eqref{eq:chat-radj-condition-new}
  by the top-right composite in~\eqref{eq:chat-radj-nat} to obtain the equivalent condition at right in~\eqref{eq:chat-radj-condition-new}.
  Now we define $\Chat_{\varpi}(\bD)$ to be the limit in $\Chat_s$ %(which exists by \cref{thm:chat-lim})
  of the diagram consisting of the objects $\incl{\varpi\circ\mu}(\cmpnt\bD\mu)$, for all $\mu : p\to r$, and the cospans
  \( \incl{\varpi\circ\nu}(\cmpnt\bD\nu) \to
    \incl{\varpi\circ\nu}(\C_\varrho(\cmpnt\bD\mu)) \leftarrow
    \incl{\varpi\circ\mu}(\cmpnt\bD\mu)
  \)
  for all $\alpha : \mu \To \nu\circ\varrho$.
  Then $\bG \to \Chat_{\varpi}(\bD)$ consists of $\overline{\cmpnt\bth\mu}$ satisfying~\eqref{eq:chat-radj-condition-new}, hence maps $\cmpnt\bth\mu$ satisfying~\eqref{eq:chat-radj-condition}.
\end{proof}

\begin{corollary}\label{thm:chat-psfr}
  We have a 2-functor $\Chat : \ladj{\M}\coop \to \cat$, with $\Chat^{\mu^\dagger} = \Chat_\mu$.
  In particular, considering only the right adjoints, we have a pseudofunctor $\Chat : \M \to \cat$.\qed
\end{corollary}

We call this pseudofunctor $\Chat$ the \textbf{co-dextrification} of $\C$.

% Note also that since $\Delta\smlock_\varpi = \Delta \lock_\varpi \smlock_{1_r}$, we also have $\Gamma\upvp \cong \Gamma\up\lock_{\varpi_*}$.

\begin{lemma}\label{thm:smlock1-psnat}
  The functors $\reflect{r} : \Chat_r \to \C_r$ are a pseudonatural transformation of pseudofunctors $\M\to\cat$.
\end{lemma}
\begin{proof}
  Let $\varpi:r\to s$ and $\bG\in \Chat_r$; we must show that $\cmpnt{\Chat_{\varpi}(\bG)}{1_s} \cong \C_\varpi(\cmpnt\bG{1_s})$.
  Since $\reflect{s} = \reflect{1_s}$ preserves $\kappa$-small limits, $\cmpnt{\Chat_{\varpi}(\bG)}{1_s}$ is the limit of the diagram consisting of the objects $\cmpnt{(\incl{\varpi\circ\mu}(\cmpnt\bG\mu))}{1_s}$, for all $\mu:p\to r$, and the analogous cospans.
  And by definition of $\incl{\varpi\circ\mu}$, each of these objects is the limit
  \[ \cmpnt{(\incl{\varpi\circ\mu}(\cmpnt\bG\mu))}{1_s} = 
    \lim_{(\sigma,\beta) \in \commacomp{(\varpi\circ\mu)}{1_s}} \C_\sigma(\cmpnt\bG\mu).
  \]
  But $\commacomp{(\varpi\circ\mu)}{1_s}$ has an initial object $(\varpi\circ\mu, 1_{\varpi\circ\mu})$, so this limit is isomorphic to $\C_{\varpi\circ\mu}(\cmpnt\bG\mu)$.
  A similar argument applies to the apices of the cospans, so $\cmpnt{\Chat_{\varpi}(\bG)}{1_s}$ is the limit of the diagram consisting of the objects $\C_{\varpi\circ\mu}(\cmpnt\bG\mu)$, for all $\mu:p\to r$, and the cospans
  \( \C_{\varpi\circ\nu}(\cmpnt\bG\nu)
    \to
    \C_{\varpi\circ\nu\circ\varrho}(\cmpnt\bG\mu)
    \leftarrow
    \C_{\varpi\circ\mu}(\cmpnt\bG\mu)
  \)
  for all $\alpha : \mu \To \nu\circ\varrho$.
  However, there is a canonical such object where $\mu=1_r$, and for any other $\mu$ the 2-cell $1_\mu : \mu \To 1_r \circ \mu$ determines a canonical cospan
  \( \C_\varpi(\cmpnt\bG{1_s}) \to \C_{\varpi\circ 1_r\circ \mu}(\cmpnt\bG\mu) \xleftarrow{=} \C_{\varpi\circ\mu}(\cmpnt\bG\mu) \)
  in which the right-hand leg is an identity.
  Thus, the limit of this diagram is isomorphic to $\C_\varpi(\cmpnt\bG{1_s})$.
  % We leave it to the reader to check the coherence axioms of this isomorphism.
\end{proof}

\begin{lemma}\label{thm:up-lax}
  The functors $\incl{r} : \C_r \to \Chat_r$ are lax natural,
  by doctrinal adjunction~\cite{kelly:doc-adjn}.\qed
\end{lemma}
% \begin{proof}
%   % , the right adjoints of any pseudo (or even colax) natural transformation form a lax natural one.
%   % Explicitly, 
%   the lax comparison map $\lock_{\mu_*} \circ (-)\up \to (-)\up \circ \C_\mu$ is the mate of the pseudofunctoriality isomorphism $\smlock_1 \circ \lock_{\mu_*} \cong \C_\mu \circ \smlock_1$.
% \end{proof}

\endgroup

\section{MATT in the co-dextrification}
\label{sec:cofree-modal-natural}

\def\incl#1{\mathsf{R}_{#1}}

We now show that for suitable $\C$, the co-dextrification $\Chat$ models MATT over $\LS$ (recall \cref{assume:ls}).
In fact, we use only its abstract properties; this makes our arguments cleaner and more general.

\subsection{Adjoint modal pre-models}
\label{sec:apm}

%The following definition abstracts all the relevant structure possessed by the co-dextrification as we constructed it in \cref{sec:cofree}, together with the necessary additional assumptions on $\C$ for the coherence theorems.
Recall that a \textbf{natural pseudo-model}~\cite[Appendix A]{shulman:univinj} is a strict natural transformation $\tau : \Tm \to \Ty$ between groupoid-valued pseudofunctors $\Tm,\Ty : \D\op \to \gpd$ that has discrete fibers and is representable.

\begin{definition}\label{def:apm}
  Let $\L$ be a 2-category with a class $\cS$ of morphisms.
  An \textbf{adjoint modal pre-model} is:
  \begin{enumerate}
  \item A modal context structure $\Chat : \ladj{\L}\coop \to \cat$, such that each $\Chat_p$ is locally cartesian closed.
    As before, we write its action on morphisms as $\Chat^\mu$, and we write $\Chat_\mu = \Chat^{\mu^\dagger}$.
    %, forming a pseudofunctor $\Chat : \L \to \cat$.
    % By restricting to the \emph{right} adjoints, we obtain a \emph{pseudofunctor} $\Chat : \L \to \cat$.
  \item A pseudofunctor $\C : \LS \to \cat$, with action on morphisms $\C_\mu$.
    % This is
    % %essentially\footnote{Strictly speaking, the universal property of $\LS$ refers to strict 2-functors, not pseudofunctors, but any pseudofunctor with codomain $\cat$ is equivalent to a strict one.}
    % equivalent to a pseudofunctor $\C: \L\to\cat$ such that $\C_\mu$ has a right adjoint if $\mu\in\cS$.
    %, whose action on morphisms in $\L$ we write as $\Cunlock\mu : \C_p\to \C_q$, with right adjoint $\Cradj\mu : \C_q \to\C_p$ whenever $\mu\in\cS$.
    %, such that each category $\C_p$ is locally cartesian closed.
  \item A pseudonatural transformation $\reflect{} : \Chat \to \C$ between pseudofunctors $\L\to\cat$.
    To be covariant on $\L$, we take the right adjoints in $\Chat$ but the left adjoints in $\C$;
    thus $\Cunlock\mu(\reflect p(\Gamma)) \cong \reflect q(\Chat_\mu(\Gamma))$.
  \item Each functor $\reflect p:\Chat_{p} \to \C_{p}$ preserves finite limits and has a fully faithful right adjoint $\incl p$.
  \item Each category $\C_p$ is a natural pseudo-model $(\C_p,\tau_p)$.
  \end{enumerate}
\end{definition}

\begin{example}
  If $\C : \L\to\cat$ is a pseudofunctor such that each $\C_p$ is locally cartesian closed with $\kappa$-small limits, each functor $\C_\mu$ preserves $\kappa$-small limits, and $\C_\mu$ has a right adjoint if $\mu\in\cS$, then the co-dextrification $\Chat$ extends it to an adjoint modal pre-model.
\end{example}

\begin{remark}\label{rmk:id-lw}
  % Although this definition is designed to capture the co-dextrifica\-tion, there are other examples.
  % In particular, 
  If each $\reflect p$ is an identity, then \cref{def:apm} is just a modal context structure $\Chat : \ladj{\L}\coop \to \cat$ consisting of locally cartesian closed natural pseudo-models such that $\Chat_{\mu}$ has a right adjoint when $\mu\in\cS$.
  In this case, the results we will prove in this section specialize to a more ordinary version of~\cite{lw:localuniv} for the modal case, when the lock functors already exist but we need to strictify the type formers.
\end{remark}

\begin{arxiv}
  Local cartesian closure will be used for $\Pi$-types, of course, but is also used in the method of~\cite{lw:localuniv} to manipulate local universes.
  To that end, we observe that $\C$ is closed under the pushforwards in $\Chat$.
\end{arxiv}

\begin{lemma}\label{thm:lei}
  In an adjoint modal pre-model, if $A\xto{f} B \xto{g}C$ are morphisms such that $f$ is a pullback of a map in the image of $\incl p$, then the pushforward $g_*(f)$ is also a pullback of a map in the image of $\incl p$.
  \begin{arxiv}
    That is, $\incl p : \C_p \to \Chat_p$ identifies $\C_p$ with a \textbf{local exponential ideal} in $\Chat_p$
  \end{arxiv}
\end{lemma}
\begin{proof}
  The pullbacks of maps in the image of $\incl p$ are a left-exact-reflective subcategory of $\Chat_p/C$; the reflection $\reflect{/C}$ applies $\reflect p$ and pulls back to $C$.
  For any $h : D\to C$, morphisms $h \to g_*(f)$ in $\Chat_p/C$ are equivalent to morphisms $g^*(h) \to f$ in $\Chat_p/B$.
  By assumption on $f$, any such morphism factors through $\reflect{/B}(g^*(h))$, which is $g^*(\reflect{/C}(h))$ by left-exactness of $\reflect p$.
  Thus, it also corresponds to a map $\reflect{/C}(h) \to g_*(f)$.
  Taking $h = g_*(f)$ we conclude that $g_*(f) \cong \reflect{/C}(g_*(f))$ and hence lies in the subcategory.
\end{proof}

\subsection{The left adjoint splitting}
\label{sec:lasplit}

% Since we assume only that $\C_p$ is a natural \emph{pseudo-}model, as happens naturally in examples, we will need to apply a coherence theorem.
The \textbf{left adjoint splitting}~\cite{lw:localuniv} of a natural pseudo-model $(\D,\tau)$ is $\taub : \Tmb \to \Tyb$ where:
\begin{itemize}
\item An element $A\in\Tyb(\Gamma)$ consists of an object $\baseof A\in \D$, a type $\topof A \in \Ty(\baseof A)$, and a morphism $\nameof{A} : \Gamma \to \baseof A$.
  We call $\baseof A$ the \emph{local universe}.
\item An element $(A,a)\in\Tmb(\Gamma)$ consists of $\baseof A\in \D$, a type $\topof A \in \Ty(\baseof A)$, and $a : \Gamma \to \baseof A\ce \topof A$.
\item The map $\taub$ sends $a$ to $\nameof A = \p_A \circ a$.
\end{itemize}
Since $\taub$ is the pullback of $\tau$ along the map $\Tyb \to \Ty$ sending $A$ to $\topof A[\nameof{A}]$, it is a natural model.
%In particular, this shows that $\taub$ is a natural model.

Given an adjoint modal pre-model, we define $\tauhatb_p = (\reflect{p})^*\taub_p$.
Thus, an element $A\in\Tyhatb_p(\Gamma)$ consists of an object $\baseof A\in \C_p$, a type $\topof A \in \Ty_p(\baseof A)$, and a morphism $\nameof{A} : \reflect{p}\Gamma \to \baseof A$, or equivalently $\nameof A : \Gamma \to \incl p\baseof A$.

\begin{lemma}\label{thm:lw-modal}
  If $(\Chat,\C)$ is an adjoint modal pre-model over $(\L,\cS)$, then $(\Chat,\tauhatb)$ is an adjoint modal natural model over $\LS$.
\end{lemma}
\begin{proof}
  The tangible morphisms in $\LS$ are $\mu\circ\nu^\dagger$, for $\mu:q\to r$ in $\L$ and $\nu : q\to p$ in $\cS$.
  Thus, we must show that in this case $(\Chat_{\nu} \circ \Chat^\mu)^* \tauhatb_p = (\reflect{p} \circ \Chat_{\nu} \circ \Chat^\mu)^* \taub_p$ is representable.
  But by pseudonaturality of $\reflect{}$, we have
  \(\reflect{p} \circ \Chat_{\nu} \circ \Chat^\mu \;\cong\; \C_\nu \circ \reflect{q} \circ \Chat^\mu
  \),
  and this has a right adjoint $\Chat_{\mu} \circ \incl q \circ \C_{\nu^\dagger}$.
  Finally, restriction along any functor with a right adjoint preserves representability.
\end{proof}

Explicitly, the comprehension $\Gamma \mnce{\mu}{\nu} A$ is the pullback
\begin{equation}
  \label{eq:mnce}
  \begin{tikzcd}%[column sep=tiny]
    \Gamma \mnce{\mu}{\nu} A \ar[d,"\phat_A"'] \ar[rr] \drpullback &&
    % (\Gamma\lock_\mu\smlock_{1_q} \ce \topof A[\nameof A]\Cradj\nu) \up \lock_{\mu_*} \ar[d]\\
    % (\Gamma\lock_\mu\smlock_{1_q}\Cunlock\nu \ce \topof A[\nameof A])\flag_\nu \up \lock_{\mu_*} \ar[d] \\
    \Chat_{\mu} \incl q  \Cradj\nu (\baseof A \ce \topof A) \ar[d] \\
    \Gamma \ar[r] &
    % \Chatunlock\mu \Chatlock\mu \Gamma \ar[r] &
    % \Chatunlock\mu \incl q \reflect q \Chatlock\mu \Gamma \ar[r] &
    \Chat_{\mu}\incl q \Cradj\nu \Cunlock\nu \reflect q \Chat^\mu(\Gamma) \ar[r,"\nameof A"'] &
    \Chat_{\mu}\incl q  \Cradj\nu \baseof A.
  \end{tikzcd}
\end{equation}
\begin{arxiv}
  Here the bottom row consists of three adjunction units followed by $\nameof A$.
\end{arxiv}

\begin{arxiv}
  It is shown by example in~\cite{lw:localuniv} that ``weakly stable'' categorical structure on a natural pseudo-model can usually be enhanced to strictly stable structure on its left adjoint splitting.
  We observe that this remains true for the mode-local type theories in an adjoint modal pre-model.
\end{arxiv}

\begin{theoremschema}\label{thm:lw}
  If $(\Chat,\C)$ is an adjoint modal pre-model, then for any of the type constructors considered in~\cite{lw:localuniv}, if $(\C,\tau)$ has weakly stable structure, then $(\Chat,\tauhatb)$ has strictly stable structure.
\end{theoremschema}
\begin{proof}
  Since $\reflect{p}$ preserves finite limits, any weakly stable or pseudo-stable structure on $\tau_p$ lifts to $(\reflect{p})^*\tau_p$.
  Therefore, by~\cite{lw:localuniv}, $((\reflect{p})^*\tau_p)^!$ has strictly stable structure.
  If we identify $\C_p$ with the image of $\incl p$, then $\Tyhatb\subseteq ((\reflect{p})^* \Ty_p)^!$ consists of the types whose local universes lie in $\C_p$.
  By \cref{thm:lei}, $\C_p$ is closed under all the local universe manipulations of~\cite{lw:localuniv}; hence $\tauhatb$ is closed under the strictly stable structure.
\end{proof}

For the modal type formers, the ``weakly stable'' structure exists on $\C$ alone; thus we name its structure.

\begin{definition}
  A \textbf{modal pre-model} over an adjoint mode theory $\M$ is a pseudofunctor $\C : \M \to \cat$ such that each $\C_p$ is a natural pseudo-model.
\end{definition}

\subsection[Pi-structure]{$\Pi$-structure}
\label{sec:pi-structure}

\begin{arxiv}
  To define the input for the coherence theorem for $\Pi$-structure, it is convenient to use the following notion.
\end{arxiv}

\begin{definition}
  A morphism $\delta:\Gamma\to \Delta$ in a natural pseudo-model is \textbf{type-exponentiable} if for any $B\in \Ty(\Gamma)$, the pushforward of $\Gamma\ce B$ along $\delta$ is isomorphic to a type projection $\Delta \ce \Pi(f,B)\to \Delta$.
  \begin{arxiv}
    In other words, there is some $\Pi(f,B) \in \Ty(\Delta)$ and a distributivity pullback
    \[
      \begin{tikzcd}
        \bullet \ar[r] \ar[d] \drpullback & \Gamma\ce B \ar[r,"\p_B"] & \Gamma \ar[d,"\delta"] \\
        \Delta\ce \Pi(f,B) \ar[rr,"\p_{\Pi(f,B)}"'] & {} & \Delta.
      \end{tikzcd}
    \]
  \end{arxiv}
\end{definition}

\begin{definition}\label{defn:pre-pi}
  A modal pre-model $\C$ has \textbf{pre-$\Pi$-structure} if for any \emph{sharp} $\mu:p\to q$ in $\M$ and any $\Gamma\in \C_p$ and $A\in \Ty_p(\Gamma)$, any pullback of
  $ \C_\mu\p_A  :   \C_\mu(\Gamma\ce A) \to  \C_\mu \Gamma$
  is type-exponentiable.
\end{definition}

\begin{arxiv}
  We will also need the following general fact.
\end{arxiv}

\begin{lemma}\label{thm:pf}
  Let $\mathsf{L}: \A \rightleftarrows \B : \mathsf{R}$ be an adjunction where $\mathsf{L}$ preserves pullbacks.
  Let $f : A \to B$ be in $\A$, $g:C\to \mathsf{L}A$ in $\B$, and suppose that the pushforward $(\mathsf{L}f)_* g : (\mathsf{L}f)_*C \to \mathsf{L}B$ of $g$ along $\mathsf{L}f$ exists in $\B$.
  Then the pullback of $\mathsf{R}((\mathsf{L}f)_* C)$ to $B$ is a pushforward along $f$ of the pullback of $\mathsf{R}g$ to $B$.
\end{lemma}
\begin{proof}
  This is a fairly straightforward diagram chase.%
\begin{arxiv}
  The desired conclusion can be rephrased as follows.
  Suppose we have a distributivity pullback in $\B$:
  \[\begin{tikzcd}
      P \ar[d]\ar[r] \drpullback & C \ar[r,"g"] & \mathsf{L}A \ar[d,"\mathsf{L}f"] \\
      (\mathsf{L}f)_*C \ar[rr] & {} & \mathsf{L}B
    \end{tikzcd}\]
  Then if we apply $\mathsf{R}$ and take pullbacks as shown in the following cube:
  % https://tikzcd.yichuanshen.de/#N4Igdg9gJgpgziAXAbVABwnAlgFyxMJZAVgBoBGAXVJADcBDAGwFcYkQAlAGQEEQBfUuky58hFGQDM1Ok1btuAIQFCQGbHgJEALKQBMMhizaIQywcI1idpAAyG5JkHwtqRm8ckkUHxhQGEVS1EtFHJSaRojeVMOAAo4rgAzAEoAfQAqOP8UlKC3K1DkcKoox3YABXz1EM89O18YkAAdZpgcegA9LI4c6vdrFFsGsr9TVvaurIq81xqPImGDUaaJju64+MTUzOzc2ZkYKABzeCJQJIAnCABbJDIQHAgkclcr26QAThon+7fru6IAAcP2eiE+-w+iAAbKC-qp3oCQY8wdDIYDYSikNp0Uhkb9EAB2XFEuGIHEIgHYsnDWRjTjHfKIpCEsmSEneLGIWmMegAIxgjAqA1CIEYMCSOBAKyca3oTKpiE5BL0JO+XNVlKh9S54TpqzaHQVUNpBL10Sc3CSAko-CAA
  \[\begin{tikzcd}[column sep={1.9cm,between origins}]
      \eta^*(\mathsf{R}P) \arrow[dd] \arrow[rd] \arrow[rr] % \drpullback[drrr] \drpullback[ddr]
      & & \eta^*(\mathsf{R}C) \arrow[rd] \arrow[rr,"\eta^*(\mathsf{R}g)"] %\drpullback[drrr]
      & & A \arrow[rd, "\eta"] \arrow[dd,"f",near end] & \\
      & \mathsf{R}P \arrow[rr] %  \drpullback
      & {} & \mathsf{R}C \arrow[rr, "\mathsf{R}g",near start,crossing over] & & \mathsf{R}\mathsf{L}A \arrow[dd, "\mathsf{R}\mathsf{L}f"] \\
      \eta^*(\mathsf{R}((\mathsf{L}f)_*(C)))
      \arrow[rd] \arrow[rrrr] % \drpullback[drr]
      & {} & {} & & B \arrow[rd, "\eta"] & \\
      & \mathsf{R}((\mathsf{L}f)_*(C)) \arrow[rrrr] \arrow[from=uu,crossing over] & {} & & & \mathsf{R}\mathsf{L}B
    \end{tikzcd}\]
  (where every face except the left- and right-hand ones is a pullback), the claim is that the back face is a distributivity pullback in $\A$.
  To show this, suppose given a pullback around $(f,\eta^*(\mathsf{R}g))$:
  \[\begin{tikzcd}[column sep=large]
      X \ar[r] \ar[d] \drpullback & \eta^*(\mathsf{R}C) \ar[r,"\eta^*(\mathsf{R}g)"] & A \ar[d,"f"] \\
      Y \ar[rr] & {} & B.
    \end{tikzcd}\]
  Since the top and bottom faces of the above cube are pullbacks, as is the right-hand half of the top face, it will suffice to show that there are unique dashed arrows completing the following cube:
  \[\begin{tikzcd}[column sep={1.9cm,between origins}]
      X \arrow[dd] \arrow[rd,dashed] \ar[rrrr] \arrow[rrrd] % \drpullback[drrr] \drpullback[ddr]
      & &  %\drpullback[drrr]
      & & A \arrow[rd, "\eta"] \arrow[dd,"f",near end] & \\
      & \mathsf{R}P \arrow[rr] %  \drpullback
      & {} & \mathsf{R}C \arrow[rr, "\mathsf{R}g",near start,crossing over] & & \mathsf{R}\mathsf{L}A \arrow[dd, "\mathsf{R}\mathsf{L}f"] \\
      Y \arrow[rd,dashed] \arrow[rrrr] % \drpullback[drr]
      & {} & {} & & B \arrow[rd, "\eta"] & \\
      & \mathsf{R}((\mathsf{L}f)_*(C)) \arrow[rrrr] \arrow[from=uu,crossing over] & {} & & & \mathsf{R}\mathsf{L}B.
    \end{tikzcd}\]
  Transposing across the adjunction $\mathsf{L}\dashv \mathsf{R}$, this becomes
  \[\begin{tikzcd}[column sep={1.9cm,between origins}]
      \mathsf{L}X \arrow[dd] \arrow[rd,dashed] \ar[rrrr] \arrow[rrrd] % \drpullback[drrr] \drpullback[ddr]
      & &  %\drpullback[drrr]
      & & \mathsf{L}A \arrow[rd,equals] \arrow[dd,"\mathsf{L}f",near end] & \\
      & P \arrow[rr] %  \drpullback
      & {} & C \arrow[rr, "g",near start,crossing over] & & \mathsf{L}A \arrow[dd, "\mathsf{L}f"] \\
      \mathsf{L}Y \arrow[rd,dashed] \arrow[rrrr] % \drpullback[drr]
      & {} & {} & & \mathsf{L}B \arrow[rd,equals] & \\
      & (\mathsf{L}f)_*(C) \arrow[rrrr] \arrow[from=uu,crossing over] & {} & & & \mathsf{L}B.
    \end{tikzcd}\]
  But since $\mathsf{L}$ preserves pullbacks, the back face of this is a pullback around $(\mathsf{L}f,g)$, so by assumption the unique dashed arrows exist.
\end{arxiv}
\end{proof}

\begin{theorem}\label{thm:chat-pi}
  If $(\Chat,\C)$ is an adjoint modal pre-model over $(\L,\cS)$ such that $\C$ has pre-$\Pi$-structure over $\LS$, then $(\Chat,\tauhatb)$ has $\Pi$-structure over $\LS$.
\end{theorem}
\begin{proof}
  Suppose we have $\mu:q\to r$ in $\L$ and $\nu : q \to p$ in $\cS$, and also $\Gamma\in \Chat_r$ and $A\in \Tyhatb_p(\Chat_\nu \Chat^\mu \Gamma) = \Ty_p(\reflect p\Chat_\nu \Chat^\mu \Gamma)$ with $B\in \Tyhatb_r(\Gamma\mnce{\mu}{\nu} A) = \Ty_r(\reflect r (\Gamma\mnce{\mu}{\nu} A))$.
  Applying $\reflect r$ to the defining pullback~\eqref{eq:mnce} of $\Gamma\mnce{\mu}{\nu} A$, and using pseudonaturality and the fact that $\reflect q \incl q \cong 1$, we have a pullback
  \begin{equation}
    \label{eq:chat-pi-1}
    \begin{tikzcd}
      \reflect r(\Gamma\mnce{\mu}{\nu} A) \ar[r] \ar[d,"\reflect r (\phat_A)"'] \drpullback &
      \Cunlock\mu\Cradj\nu(\baseof A \ce \topof A) \ar[d,"\Cunlock\mu\Cradj\nu(\p_{\topof A})"] \\
      \reflect r(\Gamma) \ar[r] &
      \Cunlock\mu\Cradj\nu(\baseof A).
    \end{tikzcd}
  \end{equation}
  Thus, \cref{defn:pre-pi} says $\reflect r (\phat_A)$ is type-exponentiable, hence the pushforward of $B$ along it is a type projection; it remains to construct a local universe making it strictly stable.
  Let $\baseof{\Pi(A,B)}$ be the universal object with maps
  $\pi_A : \baseof{\Pi(A,B)} \to \C_\omega(\baseof A)$ and $\pi_B : \pi_A^*(\C_\omega(\baseof A\ce \topof A)) \to \baseof B$.
  % \begin{align*}
  %   \pi_A &: \baseof{\Pi(A,B)} \to \C_\omega(\baseof A)\\
  %   \pi_B &: \pi_A^*(\C_\omega(\baseof A\ce \topof A)) \to \baseof B.
  % \end{align*}
  \begin{arxiv}%
    (As in~\cite{lw:localuniv}, this can be constructed using the locally cartesian closed structure of $\C_q$.)
  \end{arxiv}%
  By \cref{defn:pre-pi}, $\pi_A^*(\C_\omega(\p_{\topof A}))$ is type-exponentiable, so the pushforward of $\topof B [\pi_B] \in \Ty_q(\pi_A^*(\C_\omega(\baseof A\ce \topof A)))$ along it is represented by a type $\topof {\Pi(A,B)} \in \Ty_q(\baseof{\Pi(A,B)})$.%
  \begin{arxiv}
    That is, we have a distributivity pullback
    \[
      \begin{tikzcd}[column sep=small]
        \bullet \ar[r] \ar[d] \drpullback &
        \pi_A^*(\C_\omega(\baseof A\ce \topof A)) \ce\topof B [\pi_B] \ar[r] &
        \pi_A^*(\C_\omega(\baseof A\ce \topof A)) \ar[d] \\
        \baseof{\Pi(A,B)} \ce \topof{\Pi(A,B)} \ar[rr] & {} & \baseof{\Pi(A,B)}.
      \end{tikzcd}
    \]
  \end{arxiv}%
  Now the bottom map in~\eqref{eq:chat-pi-1} and $\nameof B : \reflect q(\Gamma\mce{\omega} A) \to \baseof B$ induce a map $\nameof{\Pi(A,B)} : \reflect q \Gamma \to \baseof{\Pi(A,B)}$.
  Together, these data define $\Pi(A,B) \in \Tyhatb_p(\Gamma)$, such that $\reflect p \Gamma \ce \topof{\Pi(A,B)}[\nameof{\Pi(A,B)}]$ is a pushforward of $\reflect q (\Gamma\mce{\omega}A) \ce \topof B[\nameof B]$ along $\reflect q (\Gamma\mce{\omega}A) \to \reflect q\Gamma$.%
  \begin{arxiv}
    Moreover, the pullback-stability of pushforwards (i.e.\ the Beck-Chevalley condition for dependent products) yields a distributivity pullback
    \[
      \begin{tikzcd}
        \bullet \ar[r] \ar[d] \drpullback &
        \reflect q (\Gamma\mce{\omega}A) \ce \topof B[\nameof B] \ar[r] &
        \reflect q (\Gamma\mce{\omega}A) \ar[d] \\
        \reflect q\Gamma \ce \topof{\Pi(A,B)}[\nameof{\Pi(A,B)}] \ar[rr] & {} & \reflect q\Gamma.
      \end{tikzcd}
    \]
  \end{arxiv}%
  The comprehension $\Gamma \ce \Pi(A,B)$ in $\Chat_q$ is defined by applying $\incl q$ to this and pulling back along the unit $\Gamma \to \incl q \reflect q \Gamma$.
  Thus, \cref{thm:pf} implies%
  \begin{arxiv}
    we have a distributivity pullback
  \[
    \begin{tikzcd}
      \bullet \ar[r] \ar[d] \drpullback &
      \Gamma\mce{\omega}A \ce B \ar[r] &
      \Gamma\mce{\omega}A \ar[d] \\
      \Gamma \ce \Pi(A,B) \ar[rr] & {} & \Gamma,
    \end{tikzcd}
  \]
  giving
  \end{arxiv}%
  the desired universal property of $\Pi(A,B)$.
\end{proof}

\subsection{Positive modalities}
\label{sec:pos-mod}

%As before, it is convenient to formulate the input using an auxiliary definition.

\begin{definition}\label{def:anodyne}
  In a natural pseudo-model, a map $f:\Gamma\to\Delta$ is \textbf{anodyne} if for any $B\in \Ty(\Delta)$ and any $g:\Gamma \to \Delta\ce B$ lifting $f$, there exists a diagonal filler:
  \[
    \begin{tikzcd}
      \Gamma \ar[d,"f"'] \ar[r,"g"] & \Delta\ce B \ar[d,"\p_B"] \\
      \Delta \ar[r,equals] \ar[ur,dashed] & \Delta
    \end{tikzcd}
  \]
  A map is \textbf{stably anodyne} if any pullback of it is anodyne.
\end{definition}

\begin{arxiv}
  In homotopical terms, the anodyne maps are those with the left lifting property against all type projections.
\end{arxiv}

\begin{definition}\label{def:pre-pos}
  A modal pre-model $\C$ has \textbf{positive pre-modalities} if for any \emph{sharp} $\mu:p\to q$ and $\Gamma\in \C_p$ with $A\in \Ty_p(\Gamma)$, there exists $\preF\mu A \in \Ty_q(\Cunlock\mu\Gamma)$ and a map
$i^\mu_{\Gamma,A} : \Cunlock\mu(\Gamma\ce A) \to \Cunlock\mu\Gamma \ce (\preF{\mu}{A})$ over $\Cunlock\mu\Gamma$.
  % \[
  %   \begin{tikzcd}[column sep=small]
  %     \Cunlock\mu(\Gamma\ce A) \ar[dr] \ar[rr,dashed,"i^\mu_{\Gamma,A}"] & & \Cunlock\mu\Gamma \ce (\preF{\mu}{A}) \ar[dl] \\
  %     & \Cunlock\mu\Gamma
  %   \end{tikzcd}
  % \]
  such that for any \emph{transparent} $\varrho:q\to r$, the map $\Cunlock\varrho(i^\mu_{\Delta,A})$ is stably anodyne.
\end{definition}

% Again, we need a general lemma.

\begin{lemma}\label{thm:smlock-anodyne}
  In an adjoint modal pre-model,
  let $\theta : \Gamma \to \Delta$ be a map in $\Chat_p$.
  If $\reflect p\theta$ is anodyne in $\C_p$, then $\theta$ is anodyne in $\Chat_p$.
\end{lemma}
\begin{proof}
  Suppose given $B\in \Tyhatb_p(\Delta) = \Tyb_p(\reflect p\Delta)$, and a commutative square as at left below. % in which we want to find a diagonal filler.
  \[
    \begin{tikzcd}
      \Gamma \ar[d,"\theta"'] \ar[r] &
      \Delta\ce B \ar[d,"\p_B"] \ar[r] \drpullback &
      \incl p(\baseof B\ce \topof B) \ar[d,"\incl p \p_{\topof B}"]\\
      \Delta \ar[r,equals] \ar[ur,dashed] &
      \Delta \ar[r,"\nameof B"'] &
      \incl p\baseof B
    \end{tikzcd}
    \hspace{2cm}
    \begin{tikzcd}
      \reflect p\Gamma \ar[d,"\reflect p\theta"'] \ar[r] &
      \baseof B\ce \topof B \ar[d,"\p_{\topof B}"]\\
      \reflect p\Delta \ar[r,"\nameof B"'] \ar[ur,dashed] &
      \baseof B
    \end{tikzcd}
  \]
  % By the universal property of pullback,
  It suffices to find a filler for the outer rectangle at left above; and by adjunction, this is equivalent to finding a filler in the square at right above.
  But such a filler exists precisely because $\reflect p\theta$ is anodyne.
\end{proof}

\begin{theorem}\label{thm:chat-pos}
  If $(\Chat,\C)$ is an adjoint modal pre-model over $(\L,\cS)$ such that $\C$ has positive pre-modalities over $\LS$, then $(\Chat,\tauhatb)$ has positive modalities over $\LS$.
\end{theorem}
\begin{proof}
  The sharp morphisms in $\LS$ are $\mu\circ\nu^\dagger$, where $\mu:q\to r$ is in $\L$ and $\nu:q\to p$ is in $\cS$.
  Suppose given these and also $\Gamma\in \Chat_r$ and $A\in \Tyhatb_p(\Chat^{\mu\circ\nu^\dagger} \Gamma) = \Tyb_p(\reflect p(\Chat^{\mu\circ\nu^\dagger}\Gamma))$, hence $\topof{A} \in \Ty_p(\baseof A)$ with
  $\nameof{A} : \reflect p \Chat^{\mu\circ\nu^\dagger}(\Gamma) \to \baseof{A}$.
  By \cref{def:pre-pos}, we have $\preF{\mu\tcirc\nu^\dagger}{\topof E} \in \Ty_q(\Cunlock\mu\Cradj\nu\baseof A)$ and
  \( i^{\mu\circ\nu^\dagger}_{\baseof A,\topof A} : \Cunlock\mu\Cradj\nu(\baseof A\ce \topof A)\to \Cunlock\mu\Cradj\nu\baseof A \ce (\preF{\mu\tcirc\nu^\dagger}{\topof A}) \)
  over $\Cunlock\mu\Cradj\nu\baseof A$,
  % \[
  %   \begin{tikzcd}[column sep=small]
  %     \Cunlock\mu\Cradj\nu(\baseof A\ce \topof A) \ar[dr] \ar[rr,dashed,"i^{\mu\circ\nu^\dagger}_{\baseof A,\topof A}"] & &
  %     \Cunlock\mu\Cradj\nu\baseof A \ce (\preF{\mu\tcirc\nu^\dagger}{\topof A}) \ar[dl] \\
  %     & \Cunlock\mu\Cradj\nu\baseof A
  %   \end{tikzcd}
  % \]
  such that $\Cunlock\varrho(i^{\mu\circ\nu^\dagger}_{\baseof A,\topof A})$ is stably anodyne for any transparent $\varrho$.
  We define $\baseof{\F{\mu\tcirc\nu^\dagger} A} = \Cunlock\mu\Cradj\nu\baseof A$ and $\topof{\F{\mu\tcirc\nu^\dagger} A} = \preF{\mu\tcirc\nu^\dagger}{\topof E}$.

  Now $\nameof{A} : \reflect p \Chat^{\mu\circ\nu^\dagger}(\Gamma) \cong \reflect p  \Chat_\nu \Chat^\mu(\Gamma) \cong \C_\nu \reflect q \Chat^\mu\Gamma \to \baseof{A}$ has an adjunct
  $\Gamma \to  \Chat_\mu \incl q \C_{\nu^\dagger}\baseof A$ (which we will sometimes denote also by $\nameof A$).
  Composing this with the lax naturality constraint of $\incl{}$, we get
  \[\Gamma \xto{\nameof A}  \Chat_\mu \incl q \C_{\nu^\dagger}\baseof A  \xto{\incl \mu} \incl r \Cunlock\mu \C_{\nu^\dagger}\baseof A = \incl r \baseof{\F{\mu\tcirc\nu^\dagger} A},\]
  whose adjunct $\reflect r \Gamma \to \baseof{\F{\mu\tcirc\nu^\dagger} A}$ we take as $\nameof{\F{\mu\tcirc\nu^\dagger} A}$.
  Finally, we define $j^{\mu\circ\nu^\dagger}_{\Gamma,A}$
  % : \Gamma\mce{\mu\circ\nu^\dagger} A \to \Gamma\ce (\F{\mu\tcirc\nu^\dagger} A)
  to make the diagram in \cref{fig:pos-mod-apm-intro} commute.
  This uses the universal property of the front rectangle as a pullback.
  Since $i^{\mu\circ\nu^\dagger}_{\baseof A,\topof A}$ is fixed along with the local universe, this definition of $j$ is strictly stable.
  This completes \cref{defn:pos}\cref{item:pos-form-intro}.

  Note that the left-hand square in back above is also a pullback (defining $\Gamma \mce{\mu\circ\nu^\dagger}A$), but the right-hand one is not: it is naturality of the lax constraint for $\incl{}$.
  However, since $\reflect r$ inverts this constraint, that square also becomes a pullback upon application of $\reflect r$.
  Thus, $\reflect r (j^{\mu\circ\nu^\dagger}_{\Gamma,A})$ is a pullback of $i^{\mu\circ\nu^\dagger}_{\baseof A,\topof A}$.

  For \cref{item:pos-elim-comp} of \cref{defn:pos}, let $\varrho : r\to s$ be in $\L$ (hence transparent in $\LS$), and suppose we have $A\in \Ty_p(\Chat^{\varrho\circ\mu\circ\nu^\dagger}(\Gamma))$.
  (Thus, the $\Gamma$ in the preceding proof of part~\cref{item:pos-form-intro} is now $\Chat^\varrho(\Gamma)$.)
  We first observe that in an adjoint modal pre-model, the construction of $\ell$ in \cref{defn:pos} is equivalent to the diagram in \cref{fig:pos-mod-apm-elim-1}, in which the map $k$ and the square \((\ast)\) are defined by the diagram in \cref{fig:pos-mod-apm-elim-2}.
  \begin{figure}
    \centering
    \begin{subfigure}{1.0\linewidth}
      \begin{center}
        \begin{tikzpicture}[->,xscale=1.4]
          \node (G) at (0,0) {$\Gamma$};
          \node (LmRCnV) at (3,0) {$\Chat_{\mu}\incl q \Cradj\nu \baseof A$};
          \draw (G) -- node[auto,swap] {$\scriptstyle \nameof A$} (LmRCnV);
          \node (RCmCnV) at (7,0) {$\incl r \Cunlock\mu\Cradj\nu \baseof A $};
          \draw (LmRCnV) -- node[auto,swap] {$\scriptstyle\incl \mu$} (RCmCnV);
          \node (GFmA) at (1,1.7) {$\Gamma\ce (\F{\mu\tcirc\nu^\dagger} A)$};
          \draw (GFmA) -- (G);
          \node (GmA) at (-1,3) {$\Gamma\mce{\mu\circ\nu^\dagger} A$};
          \draw (GmA) -- (G);
          \draw[dashed] (GmA) -- node[preaction={fill,white},inner sep=0pt] {$\scriptstyle j^{\mu\circ\nu^\dagger}_{\Gamma,A}$} (GFmA);
          \node (LmRCnVE) at (2,3) {$\Chat_{\mu}\incl q\Cradj\nu(\baseof A\ce \topof A)$};
          \draw (GmA) -- (LmRCnVE);
          \draw (LmRCnVE) -- (LmRCnV);
          \node (RCmCnVE) at (6,3) {$\incl r \Cunlock\mu\Cradj\nu(\baseof A\ce \topof A)$};
          \draw (RCmCnVE) -- (RCmCnV);
          \node[fill,white] at (8,1.7) {$\incl r (\Cunlock\mu\Cradj\nu\baseof A \ce (\preF{\mu\tcirc\nu^\dagger}{\topof A}))$};
          \node (RCmCnVmnE) at (8,1.7) {$\incl r (\Cunlock\mu\Cradj\nu\baseof A \ce (\preF{\mu\tcirc\nu^\dagger}{\topof A}))$};
          \draw (RCmCnVE) -- node[auto,near end] {$\scriptstyle \incl r (i^{\mu\circ\nu^\dagger}_{\baseof A,\topof A})$} (RCmCnVmnE);
          \draw (RCmCnVmnE) -- (RCmCnV);
          \draw[white,line width=3mm] (GFmA) -- (RCmCnVmnE);
          \draw (GFmA) -- (RCmCnVmnE);
          \draw (LmRCnVE) -- (RCmCnVE);
        \end{tikzpicture}
      \end{center}
      \caption{Introduction rule}
      \label{fig:pos-mod-apm-intro}
    \end{subfigure}
    \begin{subfigure}{1.0\linewidth}
      \begin{multline*}
        \begin{tikzcd}[ampersand replacement=\&]%[column sep=small]
          \Gamma \mce{\varrho\circ{\mu\circ\nu^\dagger}} A \ar[r,dashed,"\ell"] \ar[d] \&
          \Gamma\mce{\varrho} (\F{\mu\tcirc\nu^\dagger} A) \ar[r] \ar[d]\drpullback \&
          \Chat_{\varrho}(\Chat^\varrho(\Gamma)\ce (\F{\mu\tcirc\nu^\dagger} A)) \ar[d] \\
          \Gamma \ar[r,equals] \&
          \Gamma \ar[r,"\eta"'] \&
          \Chat_{\varrho}\Chat^\varrho(\Gamma)
        \end{tikzcd}
        \\ = \quad
        \begin{tikzcd}[ampersand replacement=\&]%[column sep=small]
          \Gamma \mce{\varrho\circ{\mu\circ\nu^\dagger}} A \ar[r,"k"] \ar[d] \ar[dr,phantom,"\scriptstyle(\ast)"] \&
          \Chat_\varrho(\Chat^\varrho(\Gamma) \mce{{\mu\circ\nu^\dagger}} A)  \ar[d] \ar[rr,"\Chat_\varrho(j^{\mu\circ\nu^\dagger}_{\Chat^\varrho(\Gamma),A})"] \&\&
          \Chat_\varrho(\Chat^\varrho(\Gamma)\ce (\F{\mu\tcirc\nu^\dagger} A)) \ar[d] \\
          \Gamma \ar[r,"\eta"'] \&
          \Chat_\varrho\Chat^\varrho(\Gamma) \ar[rr,equals] \&\&
          \Chat_\varrho\Chat^\varrho(\Gamma),
        \end{tikzcd}
      \end{multline*}
      \caption{Elimination rule, part 1}
      \label{fig:pos-mod-apm-elim-1}
    \end{subfigure}
    %   where the map $k$ and the square \((\ast)\) are defined by
    \begin{subfigure}{1.0\linewidth}
      \begin{equation*}
        \begin{tikzcd}[ampersand replacement=\&,column sep=small]
          \Gamma \mce{\varrho\circ{\mu\circ\nu^\dagger}} A \ar[r,dashed,"k"] \ar[d] \ar[dr,phantom,"\scriptstyle(\ast)"] \&
          \Chat_\varrho(\Chat^\varrho(\Gamma) \mce{{\mu\circ\nu^\dagger}} A)  \ar[d] \ar[r] \drpullback \&
          \Chat_\varrho\Chat_\mu\incl q \Cradj\nu (\baseof A \ce \topof A) \ar[d] \\
          \Gamma \ar[r,"\eta"'] \&
          \Chat_\varrho\Chat^\varrho\Gamma \ar[r,"\Chat_\varrho(\nameof A)"'] \&
          \Chat_\varrho\Chat_\mu\incl q \Cradj\nu \baseof A
        \end{tikzcd}
        =
        \begin{tikzcd}[ampersand replacement=\&,column sep=small]
          \Gamma \mce{\varrho\circ{\mu\circ\nu^\dagger}} A \ar[r] \ar[d] \drpullback \&
          \Chat_{\mu\circ\varrho}\incl q \Cradj\nu (\baseof A \ce \topof A) \ar[d] \\
          \Gamma \ar[r,"\nameof A "'] \&
          \Chat_{\mu\circ\varrho}\incl q \Cradj\nu \baseof A.
        \end{tikzcd}
      \end{equation*}
      \caption{Elimination rule, part 2}
      \label{fig:pos-mod-apm-elim-2}
    \end{subfigure}
    \caption{Positive modalities in an adjoint modal pre-model}
    \label{fig:pos-mod-apm}
  \end{figure}
  Then \((\ast)\) is a pullback, so $\ell$ is a pullback of $\Chat_\varrho(j^{\mu\circ\nu^\dagger}_{\Chat^\varrho(\Gamma),A})$.
  Hence $\reflect s(\ell)$ is a pullback of $\reflect s \Chat_\varrho(j^{\mu\circ\nu^\dagger}_{\Chat^\varrho(\Gamma),A})$, which is isomorphic to $\Cunlock\varrho \reflect r(j^{\mu\circ\nu^\dagger}_{\Chat^\varrho(\Gamma),A})$.
  But we observed that $\reflect r(j^{\mu\circ\nu^\dagger}_{\Chat^\varrho(\Gamma),A})$ is a pullback of $i^{\mu\circ\nu^\dagger}_{\baseof A,\topof A}$; thus $\reflect s(\ell)$ is also a pullback of $\Cunlock\varrho(i^{\mu\circ\nu^\dagger}_{\baseof A,\topof A})$.
  By \cref{def:pre-pos}, $\Cunlock\varrho(i^{\mu\circ\nu^\dagger}_{\baseof A,\topof A})$ is stably anodyne; hence $\reflect s(\ell)$ is anodyne.
  Thus the fillers required by \cref{item:pos-elim-comp} exist; we make them strictly stable as in~\cite[Lemmas 3.4.1.4 and 3.4.3.2]{lw:localuniv}.
\end{proof}

\subsection{Negative modalities}
\label{sec:neg-mod}

%In this case, we need to refer to the right adjoints, so even the pre-notion must be formulated with respect to $\LS$ rather than an arbitrary $\M$.

\begin{definition}\label{def:pre-neg}
  A modal pre-model $\C$ has \textbf{negative pre-modal\-ities} if for any \emph{sinister} $\mu:p\to q$, and $\Gamma\in \C_q$ with $A\in \Ty_q(\Gamma)$, we have $\preU\mu A \in \Ty_p(\Cradj\mu\Gamma)$ such that $\Cradj\mu\Gamma\ce(\preU\mu A) \cong \Cradj\mu(\Gamma\ce A)$ over $\Cradj\mu\Gamma$.
  % the functor $\Cunlock\mu : \C_p\to \C_q$ has a dependent right adjoint.
\end{definition}

% Recall that by definition of an adjoint modal pre-model (\cref{def:apm}), each such $\Cunlock\mu$ has an \emph{ordinary} right adjoint $\Cradj\mu$.
% We thus observe the following interaction between adjoints and dependent adjoints.

% \begin{lemma}
%   If $F:\C\to\D$ has both a right adjoint $G$ and a dependent right adjoint $\hat{G}$, then for any $\Gamma\in \C$ and $A\in \Ty_D(F\Gamma)$ we have a pullback square
%   \[
%     \begin{tikzcd}
%       \Gamma \ce \hat{G}A \ar[r] \ar[d] \drpullback & G(F\Gamma \ce A) \ar[d]\\
%       \Gamma \ar[r,"\eta"'] & GF\Gamma
%     \end{tikzcd}
%   \]
% \end{lemma}
% \begin{proof}
%   By the universal property of a dependent right adjoint (and of comprehension), lifts of $\gamma : \Delta \to \Gamma$ to $\Gamma \ce \hat{G} A$ are naturally bijective to lifts of $F\gamma : F\Delta \to F\Gamma$ to $F\Gamma\ce A$.
%   But by the universal property of a right adjoint, these are naturally bijective to lifts of the composite $\Delta \xto{\gamma} \Gamma \xto{\eta}GF\Gamma$ to $G(F\Gamma\ce A)$, which is the universal property of the above pullback.
% \end{proof}

\begin{theorem}\label{thm:chat-neg}
  If $(\Chat,\C)$ is an adjoint modal pre-model over $(\L,\cS)$ such that $\C$ has negative pre-modalities over $\LS$, then $(\Chat,\tauhatb)$ has negative modalities over $\LS$.
\end{theorem}
\begin{proof}
  Let $\mu:p\to q$ be in $\cS$, and $\Gamma\in \Chat_p$ with $A\in \Tyhatb_q(\Chat_\mu\Gamma) = \Tyb_q(\reflect q \Chat_\mu\Gamma)$.
  Thus, we have $\topof A \in \Ty_q(\baseof A)$ and
  \(\nameof A : \Cunlock\mu\reflect p \Gamma \cong \reflect q \Chat_\mu\Gamma \to \baseof A\).
  By assumption, we have $\preU\mu{\topof A} \in \Ty_p(\Cradj\mu\baseof A)$ and
  \(\Cradj\mu\baseof A \ce (\preU\mu{\topof A}) \cong \Cradj\mu(\baseof A \ce \topof A)\)
  over $\Cradj\mu{\baseof A}$.
  We define $\baseof{\U\mu A} = \Cradj\mu \baseof A$ and $\topof{\U\mu A} = \preU\mu {\topof A}$, and let $\nameof{\U\mu A} : \reflect p \Gamma \to \Cradj\mu \baseof A$ be the adjunct of $\nameof A$ under $\Cunlock\mu\dashv \Cradj\mu$.
  This defines $\U\mu A \in \Tyhatb_p(\Gamma)$; we must show $\Gamma\ce (\U\mu A) \cong \Gamma \mce{\mu_*} A$.
  Now $\Gamma\ce (\U\mu A)$ is defined by the pullback square at left below:
  \[
    \begin{tikzcd}
      \Gamma\ce (\U\mu A)  \ar[d] \ar[r]\drpullback &
      \incl p(\Cradj\mu\baseof A \ce (\preU\mu{\topof A})) \ar[d] \ar[r,"\cong"] &
      \incl p \Cradj\mu(\baseof A \ce \topof A) \ar[d] \\
      \Gamma \ar[r] &
      \incl p \Cradj\mu\baseof A \ar[r,equals] &
      \incl p \Cradj\mu\baseof A
    \end{tikzcd}
  \]
  Composing with the isomorphism on the right, we obtain the defining pullback of $\Gamma \mce{\mu_*} A$ as in~\eqref{eq:mnce}.
\end{proof}

\section{Diagrams of 1-topoi}
\label{sec:examples}

Combining \cref{thm:chat-pi,thm:chat-pos,thm:chat-neg}, we have the following.
(Recall \cref{assume:ls}.)

\begin{theorem}\label{thm:chat-matt}
  Let $\L$ be a 2-category with a class of morphisms $\cS$.
  If an adjoint modal pre-model $(\Chat,\C)$ over $(\L,\cS)$ is such that $\C$ has pre-$\Pi$-structure, positive pre-modalities, and negative pre-modalities over $\LS$, then $(\Chat,\tauhatb)$ models MATT over $\LS$.\qed
\end{theorem}

Any category with pullbacks has a canonical natural pseudo-model where all maps are type projections.

\begin{lemma}\label{thm:set-pre}
  Let $\M$ be an adjoint mode theory, and $\C:\M\to \cat$ be a pseudofunctor such that each $\C_p$ is locally cartesian closed.
  If we make $\C$ a modal pre-model in the canonical way, as above, then it has pre-$\Pi$-structure, positive pre-modalities, and negative pre-modalities.
\end{lemma}
\begin{proof}
  Since $\C_p$ is locally cartesian closed and everything is a type projection, we have pre-$\Pi$-structure.
  For positive pre-modalities we take $i^\mu_{\Gamma,A}$ to be an identity, and similarly for negative pre-modalities.
\end{proof}

% Putting everything together in this case yields the following conclusion.

\begin{theorem}\label{thm:set-model}
  Let $\kappa$ be an infinite regular cardinal, $\L$ a $\kappa$-small 2-category with a class of morphisms $\cS$, and $\C:\L\to \cat$ a pseudofunctor such that each $\C_p$ is locally cartesian closed with $\kappa$-small limits, each $\C_\mu$ preserves $\kappa$-small limits, and has a right adjoint if $\mu\in\cS$.
  Then $\Chat$ models extensional MATT over $\LS$.
\end{theorem}
\begin{proof}
  By \cref{thm:chat-lim}, local cartesian closure lifts from $\C$ to $\Chat$.
  Thus, $(\Chat,\C)$ is an adjoint modal pre-model, so \cref{thm:chat-matt,thm:set-pre} yield a model of MATT.
  Composition and diagonals yield weakly stable $\Sigma$-types and extensional identity types in each $\C_p$, hence mode-locally by \cref{thm:lw}.
\end{proof}

\begin{remark}
  In addition, the following should follow from \cref{thm:chat-lim,thm:lw}.
  \begin{itemize}
  \item If each $\C_p$ has finite coproducts, then $\Chat$ models sum types at each mode.
  \item If each $\C_p$ is locally presentable and each $\C_\mu$ is accessible, then each $\Chat_p$ is again locally presentable.
    Thus, by the methods of~\cite{ls:hits}, $\Chat$ models inductive types and quotient-inductive types at each mode.
    % The data for a (Q)IT then determines an accessible monad on some slice as in~\cite{ls:hits}, and the algebraic coproduct~\cite{kelly:transfinite} of this monad with a fiberwise version of the reflector $\reflect p$ has an initial algebra that is a type projection and determines a pseudo-stable (Q)IT, and \cref{thm:lw} again applies.
  \item If $\C$ is a diagram of Grothendieck topoi and geometric morphisms, then each $\Chat_p$ is also a topos.
    Thus, if there are enough inaccessible cardinals, $\Chat$ models universes at each mode (see~\cite{hs:lift-univ,streicher:universes,gss:univ-topoi,shulman:univinj}).
    % (See~\cite[Appendix A]{shulman:univinj} for a version of the left adjoint splitting that incorporates universes.)
  \end{itemize}
%  We leave the details to the reader.
\end{remark}

% A particularly important and interesting case of \cref{thm:set-model} is the following.
Let $\topos$ denote the 2-category of Grothendieck topoi, geometric morphisms, and transformations.
% The 2-category $\topos\coop$ is sometimes called the 2-category of \emph{logoi}.

\begin{theorem}\label{thm:topos}
  Let $\L$ be a finite 2-category and $\E : \L\coop \to \topos$ a pseudofunctor.
  Then the co-dextrification $\Ehat$ models extensional MATT over $\ladj{\L}$, with positive and negative modalities representing inverse image and direct image functors respectively, and extensional MLTT at each mode.\qed
\end{theorem}

\begin{remark}
  \cref{thm:topos} does not state explicitly how to extract conclusions about $\E$ from the interpretation of MATT in $\Ehat$.
  We will not try to make this precise here, but the idea is that $\Ehat_p$ can be viewed as a ``presentation'' of $\E_p$ via the reflector $\reflect p : \Ehat_p \to \E_p$, and that the interpretation of MATT respects this ``quotient''.
  For instance, the anodyne context morphisms (\cref{def:anodyne}) in $\Ehat_p$ are precisely those that are inverted by $\reflect p$; thus MATT is ``unable to distinguish'' contexts that present the same object of $\E_p$.
  One way to make this more precise is using Quillen model categories.
\end{remark}

\begin{table}
  \centering
  \begin{tabular}{p{3.8cm}p{5.7cm}p{1cm}p{5cm}}
    2-category $\L$ & Semantics & $\dashv$ ? & MATT related to\\\hline
    Single morphism & Geometric morphism \newline \labelitemi\ Cartesian comonad coalgebras & no\newline no &
    AdjTT~\cite{zwanziger:natmod-comnd} \newline CoTT~\cite{zwanziger:natmod-comnd}\\
    % $\mu\dashv \nu$ & Countably cocontinuous adjoint triple & no & \\
    Idempotent monad & Totally connected topos & no &
    Parametric TT~\cite{nvd:parametric-quant} \\
    Idempotent comonad &
    Local topos (\cref{eg:local}) \newline
    \labelitemi\ Johnstone's topos~\cite{ptj:topological-topos} \newline
    \labelitemi\ $\kappa$-condensed sets~\cite{scholze:lec-condensed,bh:pyknotic-i} \newline
    \labelitemi\ Cohesive topos~\cite{lawvere:cohesion,schreiber:dcct} &
    no\newline no \newline no \newline yes & Spatial TT~\cite{shulman:bfp-realcohesion}, Crisp TT~\cite{lops:internal-universes} \\
    Idem.\ monad w/ $\triangleright$ &  Topos of trees & yes &
    Guarded TT\newline \cite[\S9]{gknb:mtt} and~\cite[\S VI]{gckgb:fitchtt}\\
    Meet-semilattice &
    Commuting foci \newline
    \labelitemi\ Simplicial spaces (\cref{eg:local}) \newline
    \labelitemi\ Differential cohesive topos & no\newline no\newline yes & \cite{mr:commuting-cohesions}\newline \newline \cite{glnprsw:dctt}\\
    Idempotent bimonad & Parametrized spectra\tablefootnote{This is an $\infty$-topos without any 1-categorical analogue, so it is not covered by the semantic results in this paper.} & yes & \cite{rfl:synthetic-spectra}\\
    \cref{eg:gr} & Geometric realization & no & 
  \end{tabular}
  \caption{Instantiations of MATT and their semantics}
  \label{tab:egs}
\end{table}

We end by discussing some examples of simple classes of diagrams in $\topos$, to explore the flexibility and the limits of \cref{thm:set-model,thm:topos}.
As we will see, in some cases extra left adjoints already exist, so that co-dextrification is not necessary; but even in this case, some coherence results like those of \cref{sec:cofree-modal-natural} are often still needed (see \cref{rmk:id-lw}).
\Cref{tab:egs} summarizes some of the following examples, along with whether left adjoints already exist, and pointers to related theories in the literature.

\begin{example}\label{eg:adj}
  If $\L$ consists of two objects $p,q$ and one nonidentity morphism $\mu:p\to q$, then a functor $\L\coop\to \topos$ is a single geometric morphism.
  The resulting instance of MATT has two modes related by an adjoint pair of modalities $\F\mu\blank$ and $\U\mu\blank$.
  It is related to the split-context theory AdjTT of~\cite{zwanziger:natmod-comnd}, and can be interpreted in any geometric morphism.

  In particular, there is a unique geometric morphism from any topos $\E$ to $\Set$.
  The resulting instance of MATT combines the usual internal language of $\E$ at one mode with the classical world of $\Set$ at another mode, with a ``discrete objects'' modality $\F\mu\blank$ taking any set to an object of $\E$, having a right adjoint ``global sections'' modality $\U\mu\blank$.
  This allows us to use the internal logic of $\E$ but also make ``external'' statements when needed, e.g.\ to study the cohomology of $\E$, or ``global'' structures that do not lift to arbitrary slices.
  In the language of Lawvere~\cite{lawvere:var-qty}, terms at mode $q$ would be called ``variable quantities'', while those at at mode $p$ would be ``constant quantities''.
  Since the functor $\Set \to \E$ does not in general have a left adjoint, such an interpretation for a general topos is impossible without co-dextrification.
  (When this functor does have a left adjoint, one says that $\E$ is \emph{locally connected}.)
  % ; this generalizes local connectedness for topological spaces.)

  The composite sending $A$ to $\F\mu{\U\mu{A}}$ is then a comonad on one mode, while the one sending $B$ to $\U\mu{\F\mu A}$ is a monad on the other.
  The 2-categories freely generated by a monad or a comonad are infinite, but if a monad or comonad decomposes through a geometric morphism in this way we can internalize it in MATT without needing infinite limits.
  Such a comonad decomposes if and only if it preserves finite limits; indeed the category of coalgebras for a finitely continuous comonad on a topos is again a topos.
  (If we also identify an object with its cofree coalgebra, so we only need one mode of syntactic types, we obtain something like the CoTT of~\cite{zwanziger:natmod-comnd}.)
  The category of algebras for a finitely continuous monad on a topos need not be a topos; but if the topos is Boolean, then it can be induced by \emph{some} geometric morphism~\cite{johnstone:cart-mnd}.
  And for a general finitely continuous monad on a topos, the category of algebras is at least locally cartesian closed by~\cite{kock:bilin-ccmnd} on slice categories, so \cref{thm:set-model} can still be applied.
\end{example}

\begin{example}
  Let $\L$ be the 2-category freely generated by an adjunction $\mu\dashv \nu$.
  Then in $\ladj{\L}$ we have $\mu^\dagger \cong \nu$ by uniqueness of adjoints, so $\ladj{\L}$ is generated (up to equivalence) by an adjoint triple $\mu \dashv \nu \dashv \nu^\dagger$.
  Since $\L$ is countably infinite, we can interpret MATT over this $\ladj{\L}$ in any adjoint triple of functors between toposes (or more general categories) whose left adjoint preserves countable limits (the right adjoints do automatically, of course).
  We would generally prefer to represent the adjoint triple with the modalities $\F\mu\blank$, $\U\mu\blank$, and $\U\nu\blank$, since $\U\mu\blank$ has stronger rules than the equivalent $\F\nu\blank$.
\end{example}

\begin{example}\label{eg:subtop-clc}
  By contrast, the 2-category $\L$ freely generated by a \emph{strictly reflective} adjunction (i.e.\ whose counit is an identity) is finite.
  It has two modes $p$ and $q$, morphisms $\mu:p\to q$ and $\nu:q\to p$ such that $\mu\circ\nu = 1_q$,
% (which is thus the only endomorphism of $q$)
  and a 2-cell $\eta : 1_p \To \nu\circ\mu$
  % (which are the only two endomorphisms of $p$)
  such that $\eta\triangleright \nu = 1_\nu$ and $\mu\triangleleft\eta = 1_\mu$.
  This determines all the composites, so no additional data are needed.
  A pseudofunctor $\L\to\cat$ is a non-strictly reflective adjunction, whose counit is an isomorphism.
  Thus we can interpret MATT over this $\L$ in an arbitrary reflective adjunction between toposes, giving a modal type theory for a topos equipped with a subtopos.
  % (We can also interpret \cref{eg:adj} in such an adjunction, but its reflectivity would not be enforced judgmentally.)

  If the inclusion functor has a further right adjoint, we have a coreflective adjunction in $\topos$, and we can interpret MATT over $\ladj{\L}$ with $\nu$ sinister.
  The induced geometric morphism from the larger topos to the smaller one is then called \emph{totally connected}.
  For instance, the ``topos of trees'' used in guarded recursion theory (presheaves on $(\mathbb{N},\le)$) is totally connected over $\Set$; the modal type theories that it models are discussed in~\cite[\S9]{gknb:mtt} and~\cite[\S VI]{gckgb:fitchtt}.
  (In this topos, the left adjoint happens to already have a further left adjoint, so co-dextrification is not required to interpret modal type theory.)
\end{example}

\begin{example}\label{eg:conn-local}\label{eg:local}
  Taking $\L$ to be the opposite of the one from \cref{eg:subtop-clc}, we can interpret MATT in an arbitrary \emph{coreflective} adjunction between toposes.
  This is the same as a \emph{connected} geometric morphism, such as that from sheaves on some connected space to $\Set$.

  If the right adjoint has a further right adjoint, so that we can interpret MATT over $\ladj{\L}$, the geometric morphism is called \emph{local}.
  This property rarely holds for the topos of sheaves on a space (a ``little topos''), but it often does for toposes whose \emph{objects} can be interpreted as some kind of space (``big toposes'').
  Big toposes are a natural home for \emph{synthetic topology}.
  One is Johnstone's topological topos~\cite{ptj:topological-topos}, whose objects are a sort of sequential convergence space; the internal language of this topos is used for instance in~\cite{escardo:top-hoil}.
  A related topos is $\kappa$-condensed sets~\cite{scholze:lec-condensed}, which has been advocated for the study of algebraic objects equipped with topology.
  (When $\kappa$ is inaccessible, these are also called pyknotic sets~\cite{bh:pyknotic-i}.
  The category of all ``condensed sets'' is locally cartesian closed but not a topos.)

  Many local toposes are also \emph{cohesive}, meaning that the leftmost adjoint of their adjoint triple has a further left adjoint.
  This left adjoint is not usually finitely continuous, so it cannot be represented internally as a judgmental modality with co-dextrification, although it can be introduced axiomatically as in~\cite{shulman:bfp-realcohesion,mr:commuting-cohesions}.
  When it exists, co-dextrification is not needed to interpret the other modalities.
  But Johnstone's topological topos and $\kappa$-condensed sets are not cohesive, so co-dextrification is necessary in those cases.

  Non-cohesive local toposes turn out to have many advantages.
  One clear advantage is that they can include non-locally-connected spaces, which arise naturally in many parts of mathematics.
  In addition, they often do a better job of faithfully encoding topological notions such as unions of closed sets, constructions of cell complexes (including geometric realization of simplicial sets), and cohomology.
  In~\cite{ex:kleene-kreisel} such a topos was used to represent the Kleene--Kreisel functionals and model principles of intutionism.
\end{example}

\begin{example}
  We can simplify the mode theory of \cref{eg:conn-local} by removing the mode corresponding to the base topos.
  Then $\L$ has one mode $p$ and a single idempotent comonad $\mu:p\to p$, and is again finite.
  Thus, MATT over this $\L$ can be interpreted in any topos that is connected over some base, with the base topos visible as the modal types for the comonad $\F\mu\blank$.
  This instantiation of MATT is similar to the split-context ``crisp type theory'' used in~\cite{lops:internal-universes} to construct universes in cubical sets.
  % (TODO: Describe the co-dextrification: it's essentially a comma category over the inclusion of coalgebras.)

  If the topos is additionally local, the comonad has a right adjoint monad, and we can interpret MATT over $\ladj{\L}$.
  This is similar to the split-context ``spatial type theory'' of~\cite{shulman:bfp-realcohesion}, which was conjectured to be interpretable in any local topos; we have thus established this for a related lock-based theory.
  Note that although examples like Johnstone's topological topos and pyknotic sets require co-dextrification, the intended model of~\cite{shulman:bfp-realcohesion} is cohesive and hence does not.
\end{example}

\begin{example}
  Applying the same simplification to \cref{eg:subtop-clc}, we obtain a 2-category $\L$ with one mode and an idempotent \emph{monad}, for which $\ladj{\L}$ can be interpreted in any topos that is totally connected over some base.
  This is related to the left-lifting theory of~\cite{nvd:parametric-quant}, which is interpreted in a topos of ``bridge/path cubical sets'' that is totally connected over ordinary cubical sets.
  (The left adjoint in this case also has a further pair of left adjoints.)
\end{example}

\begin{example}
  If $\L$ is a meet-semilattice, regarded as a mon\-oid\-al poset and thereby a one-object 2-category, we obtain an instance of MATT that is similar to the left-lifting theory of~\cite{mr:commuting-cohesions}.
  In many of their examples each ``focus'' is cohesive, hence the further left adjoints needed for locks already exist.
  This includes the situation of ``differential cohesion'' studied in~\cite{glnprsw:dctt}, as well as other related situations.
  But there are related examples in which not all foci are cohesive, such as simplicial objects in the topological topos, or simplicial pyknotic sets, and for these we require co-dextrification.
\end{example}

\begin{example}
  The left-lifting theory of~\cite{rfl:synthetic-spectra} is similar to MATT over the 2-category generated by an idempotent endomorphism that is both a monad and a comonad and adjoint to itself.
  This means we can represent its modality negatively, and use it as its own lock functor in semantics, thereby interpreting this instance of MATT in any topos equipped with such an endofunctor.
  Unfortunately, the intended model of~\cite{rfl:synthetic-spectra} is an $(\infty,1)$-topos without an evident 1-categorical analogue, so it is not covered by this paper.
\end{example}

\begin{example}\label{eg:gr}
  By~\cite[Theorem 8.1]{ptj:topological-topos}, there is a geometric morphism $S : \E \to \sSet$ from Johnstone's topological topos $\E$ to the topos $\sSet$ of simplicial sets, whose direct image $S_*$ is the total singular complex (suitably generalized) and whose inverse image $S^*$ is geometric realization.
  Since both $\E$ and $\sSet$ are local over $\Set$, this allows us to reason formally about geometric realization using an instance of MATT with three modes --- say $t$ for the topological topos, $s$ for simplicial sets, and $d$ for discrete sets --- with sinister coreflective adjunctions relating $d$ to both $t$ and $s$, and a sinister morphism $\sigma : s\to t$ for the geometric realization adjunction.
  As $\E$ is not cohesive (though $\sSet$ is), and geometric realization is not a right adjoint, this would be impossible without co-dextrification.
  Using~\cite[Theorem 8.2]{ptj:topological-topos}, we can do something similar for geometric realization of ``simplicial spaces'', i.e.\ simplicial objects of $\E$.
\end{example}

\section{Conclusion and future work}
\label{sec:conclusion}

We have shown that, contrary to appearances, general modal type theories formulated with ``context locks'' following~\cite{gknb:mtt,gckgb:fitchtt} can be interpreted in diagrams of categories without requiring additional left adjoints to interpret the locks.
This significantly expands the potential semantics of such theories, strengthening the argument that they are a good general approach to modal dependent type theories.
In addition, we have formulated MATT, a general context-lock modal type theory that unifies the positive modalities of~\cite{gknb:mtt} with the negative ones of~\cite{gckgb:fitchtt}, and shown that it is the natural type theory to interpret in our semantics.

We have, however, left many open questions for future research, such as the following.
\begin{enumerate}
\item Can the assumption of $\kappa$-small limits be weakened, specifically when $\kappa>\omega$?
\item It is known~\cite{shulman:univinj} that intensional dependent type theory can be interpreted in any $(\infty,1)$-topos.
  Can intensional MATT be interpreted in any \emph{diagram} of $(\infty,1)$-topoi?
\item Is there a full ``internal language correspondence'' relating MATT to suitable diagrams of categories?
  E.g.\ do adjoint modal natural models have a homotopy theory that presents diagrams of categories?
\item Does MATT satisfy normalization, and which $(\L,\cS)$ are decidable?  (See \cref{rmk:norm}.)
\item Is there a general modal dependent type theory using left multi-liftings, and can it be interpreted in the co-dextrification?
  Can it be generalized to cases where left multi-liftings do not exist?
\item In~\cite{lsr:multi}, \emph{simple} modal type theories were unified with substructural ones.
  Is there a context-lock approach to substructurality?
  Can it be unified with modal dependent type theory?
\end{enumerate}

\begin{arxiv}
\appendix

\section{The universal property of co-dextrification}
\label{univ-prop}

Here we sketch a proof that the co-dextrification is a right adjoint.
We continue the notation of \cref{sec:cofree}.

Let $\lexk$ denote the 2-category of categories with $\kappa$-small limits, functors that preserve such limits, and natural transformations.
Let $\Ps(\M,\lexk)$ denote the 2-category of pseudofunctors, colax natural transformations, and modifications from $\M$ to $\lexk$.
(A colax natural transformation $F : \C \to \D$ consists of components $F_p : \C_p \to \D_p$ and 2-cells $F_\mu : F_q \circ \C_\mu \To \D_\mu \circ F_p$, satisfying functoriality and naturality.)

Let $\Psadj(\M\coop,\lexk)$ denote the 2-category of:
\begin{itemize}
\item Pseudofunctors $\D:\M\coop\to\lexk$, with action on morphisms and 2-cells written as $\D^\mu$ and $\D^\alpha$ respectively, and such that each $\D^\mu$ has a right adjoint $\D_{\mu}$ in $\lexk$.
\item Pseudonatural transformations between pseudofunctors $\M\coop\to\lexk$.
\item All modifications between these.
\end{itemize}

\begin{lemma}
  There is a 2-functor $U : \Psadj(\M\coop,\lexk) \to \Ps(\M,\lexk)$.
\end{lemma}
\begin{proof}
  Since passage to adjoints is pseudofunctorial, for $\D\in\Psadj(\M\coop,\lexk)$ the right adjoints $\D_{\mu}$ form a pseudofunctor $\M\to\lexk$.
  If $F:\C\to\D$ is a pseudonatural transformation between $\C,\D\in\Psadj(\M\coop,\lexk)$, then the pseudonaturality isomorphism $\D^\mu \circ F_q \cong F_p \circ \C^\mu$ has a mate $F_q \circ \C_{\mu} \To \D_{\mu} \circ F_p$, providing the 2-cells to make $F$ a colax natural transformation between pseudofunctors $\M\to\lexk$.
  % We leave the rest to the reader.
\end{proof}

Thus, \cref{sec:cofree} shows that from $\C\in\Ps(\M,\lexk)$ we have constructed $\Chat \in \Psadj(\M\coop,\lexk)$, along with a map $\reflect{} : U\Chat \to \C$ in $\Ps(\M,\lexk)$.
In fact, $\Chat$ is even a strict 2-functor $\M\coop\to\lexk$, and the map $\reflect{}$ is even \emph{pseudo} natural.
The former distinction is immaterial (any pseudofunctor into $\lexk$ is equivalent to a strict one); but regarding the latter, it appears to be important for the universal property that we use colax transformations in general even though $\reflect{}$ is pseudo.

\begin{theorem}
  Let $\C\in \Ps(\M,\lexk)$, with $\Chat$ constructed as in \cref{sec:cofree}.
  Then for any $\D\in \Psadj(\M\coop,\lexk)$, the functor
  \[ \Psadj(\M\coop,\lexk)(\D,\Chat)
    % \xto{U} \Ps(\M,\lexk)(U\D,U\Chat) \xto{(\smlock_1 \circ -)}
    \xto{(\reflect{} \circ U-)}
    \Ps(\M,\lexk)(U\D,\C) \]
  is an equivalence of categories.
  Thus, $\C \mapsto \Chat$ is bicategorically right adjoint to the forgetful functor $U:\Psadj\to\Ps$.
\end{theorem}
\begin{proof*}{Sketch of proof.}
  We construct a pseudo-inverse.
  Let $G : \D\to\C$ be colax natural; we define $\Ghat : \D\to\Chat$ as follows.
  For $\Gamma\in \D_r$ and $\mu:p\to r$, we set
  \[(\Ghat_r\Gamma)^\mu = G_p(\D^\mu(\Gamma)) \quad \in \C_p. \]
  And for $\alpha : \mu \To \nu\circ\varrho$, we let $(\Ghat_r\Gamma)^\alpha$ be the composite
  \begin{equation*}
    (\Ghat_r\Gamma)^\nu
    = G_q(\D^\nu(\Gamma))
    \xto{G_q(\overline{\D^\alpha(\Gamma)})} G_q(\D_\varrho \D^\mu(\Gamma))
    \xto{G_\varrho} \C_\varrho(G_p(\D^\mu(\Gamma)))
    = \C_{\varrho}((\Ghat_r\Gamma)^\mu)
  \end{equation*}
  where $\overline{\D^\alpha(\Gamma)}: \D^\nu(\Gamma)\to \D_\varrho \D^\mu(\Gamma)$ is the mate of the map
  \[ \D^\alpha(\Gamma) : \D^\varrho \D^\nu(\Gamma) \to \D^\mu(\Gamma) \]
  arising from pseudofunctoriality of the left adjoints for $\D$.
  All the axioms follow from the functoriality and naturality of mates, as does functoriality; thus we have a functor $\Ghat_r : \D_r \to \Chat_r$.

  Now let $\omega : r \to s$; we show $\Ghat$ commutes with $\D^\omega$ and $\Chat^\omega$, componentwise.
  We have
  \[ (\Ghat_r(\D^\omega(\Gamma)))^\mu
    = G_p(\D^\mu\D^\omega(\Gamma))
    \cong G_p(\D^{\omega\circ\mu}(\Gamma))
    = (\Ghat_s\Gamma)^{\omega\circ\mu}
    = (\Chat^\omega(\Ghat_s\Gamma))^\mu.
  \]
  If $\D$ is a strict 2-functor, this is a strict equality.
  Otherwise, the pseudonaturality axioms follow.

  On one side, if we map $\Ghat$ back down into $\Ps(\M,\lexk)(U\D,\C)$, its components are $(\Ghat_p \Gamma)^{1_p} = G_p(\D^{1_p}(\Gamma)) \cong G_p \Gamma$.
  This defines an isomorphism $G \cong \reflect{} \circ U\Ghat$ in $\Ps(\M,\lexk)(U\D,\C)$.

  On the other side, suppose $F:\D \to \Chat$ is a morphism in $\Psadj(\M\coop,\lexk)$.
  We want to compare it to $\widehat{\reflect{}\circ UF}$, so we compute using the definition of the latter:
  \[
    ((\widehat{\reflect{}\circ UF})_r (\Gamma))^\mu = \reflect{p}F_p(\D^\mu(\Gamma))
    \cong \reflect{p}({\Chat}^\mu(F_r\Gamma))
    = (F_r\Gamma)^\mu
  \]
  These define the components of an isomorphism $\widehat{\reflect{}\circ UF}\cong F$ in $\Psadj(\M\coop\lexk)$.
\end{proof*}
\end{arxiv}

\bibliographystyle{entics}
\bibliography{all}

\end{document}